\documentclass[12pt,reqno]{amsart}
\usepackage[top=1in, bottom=1in, left=1in, right=1in]{geometry}
\usepackage{times, amsthm, amssymb, amsmath, amsfonts, bm, graphicx,setspace}
\usepackage{mathrsfs,wasysym}
\usepackage[pagebackref=true,pdftex]{hyperref}
\usepackage[all]{xy}
\usepackage[usenames,dvipsnames]{color}

\newtheorem{theorem}{Theorem}[section]
\newtheorem{lemma}[theorem]{Lemma}

\newtheorem{corollary}[theorem]{Corollary}
\newtheorem{proposition}[theorem]{Proposition}

\theoremstyle{definition}
\newtheorem{example}[theorem]{Example}
\newtheorem{remark}[theorem]{Remark}

\newtheorem{definition}[theorem]{Definition}
\newtheorem{convention}[theorem]{Convention}

\newtheorem{notation}[theorem]{Notation}

\newenvironment{romanlist}%
        {\begin{list}
                {\noindent\makebox[0mm][r]{(\roman{enumi})}}
                {\leftmargin=5.5ex \usecounter{enumi}}
        }
        {\end{list}}

\DeclareMathAlphabet{\mathpzc}{OT1}{pzc}{m}{it}

\def\<{\langle}
\def\>{\rangle}

\def\CC{{\mathbb C}}
\def\calD{{\mathcal D}}

\def\NN{{\mathbb N}}
\def\OO{{\mathcal O}}
\def\calO{{\mathcal O}}

\def\QQ{{\mathbb Q}}

\def\RR{{\mathbb R}}

\def\ZZ{{\mathbb Z}}

\def\boldone{\boldsymbol{1}}
\def\boldzero{\boldsymbol{0}}
\def\sC{{\mathscr C}}

\def\sL{{\mathscr L}}
\def\sN{{\mathscr N}}
\def\sQ{{\mathscr Q}}

\def\sT{{\mathscr T}}

\def\del{\partial}
\def\dt#1{{\partial_{t_{#1}}}}
\def\dx#1{{\partial_{x_{#1}}}}

\def\dz#1{{\partial_{z_{#1}}}}

\def\ini{\operatorname{in}}
\def\Gl{\operatorname{\mathrm{GL}}}

\def\Horn{{\rm Horn}}
\def\sHorn{{\rm sHorn}}
\def\qdeg{{\rm qdeg}} 
\def\Span{{\rm Span}}
\def\hom{\operatorname{Hom}}
\def\vol{{\rm vol}}
\def\conv{{\rm conv}}
\def\rank{{\rm rank }}
\def\ann{\operatorname{ann}}
\def\supp{{\rm supp}}
\def\Var{{\rm Var}}
\def\Sol{\operatorname{Sol}}
\def\gr{{\rm gr}}
\def\image{{\rm image\ }}
\def\dkap{{\delta_{B,\kappa}}}
\DeclareMathOperator{\SingLocus}{{\operatorname{Sing}}}
\def\Sing#1{{\SingLocus({#1})}}
\def\charVar{{\operatorname{Char}}}
\DeclareMathOperator{\Dmodules}{{\operatorname{\mathfrak{D-mods}}}}
\def\Dmods#1{{\Dmodules({#1})}}
\DeclareMathOperator{\Aholonomic}{{\operatorname{{\it{A}}-\mathfrak{hol}}}}
\def\Ahol#1{{\Aholonomic({#1})}}
\DeclareMathOperator{\Amodules}{\operatorname{{\it{A}}-\mathfrak{mods}}}
\def\Amods#1{{\Amodules({#1})}}

\protect\DeclareMathOperator{\AEmodules}{\operatorname{\sT-\mathfrak{hol}}}
\def\AEmods#1{{\AEmodules({#1})}}
\DeclareMathOperator{\AEmaxmodules}{\operatorname{\sT-\mathfrak{irred}}}
\def\AEmaxmods#1{{\AEmaxmodules({#1})}}
\DeclareMathOperator{\Tirr}{\operatorname{\sT-\mathfrak{irred}}}

\def\barX{{\bar{X}}}
\def\barZ{{\bar{Z}}}

\def\onto{\twoheadrightarrow}
\def\into{\hookrightarrow}
\def\minus{\smallsetminus}
\def\dtorus{(\mathbb{C}^*)^d}

\def\endrk{\hfill$\hexagon$}

\newcommand*{\defeq}{\mathrel{\vcenter{\baselineskip0.5ex \lineskiplimit0pt
                     \hbox{\scriptsize.}\hbox{\scriptsize.}}}%
                     =}

\numberwithin{equation}{section}
\begin{document}
\vspace*{-6mm}
\title[Torus equivariant $D$-modules and hypergeometric systems]{Torus equivariant $D$-modules and hypergeometric systems}  

\author[Berkesch]{Christine Berkesch}
\address{School of Mathematics, University of Minnesota, Minneapolis, MN 55455.}
\email{cberkesc@umn.edu}
\thanks{CB was partially supported by NSF Grants DMS 1440537, OISE 0964985, DMS 0901123, and DMS 1661962.}

\author[Matusevich]{Laura Felicia Matusevich}
\address{Department of Mathematics \\
Texas A\&M University \\ College Station, TX 77843.}
\email{laura@math.tamu.edu}
\thanks{LFM was partially supported by NSF Grants DMS 0703866, DMS 1001763, and a Sloan Research Fellowship.}

\author[Walther]{Uli Walther}
\address{Department of Mathematics\\ Purdue University\\ West Lafayette, IN \ 47907.}
\email{walther@math.purdue.edu}
\thanks{UW was partially supported by NSF Grant DMS 0901123 and 
DMS 1401392.}


\begin{abstract}
We formalize, at the level of $D$-modules, the notion that
$A$-hypergeometric systems are equivariant versions of the classical
hypergeometric equations. For this purpose, we construct a functor
$\Pi^{\tilde A}_B$ on a suitable category of torus equivariant
$D$-modules and show that it preserves key properties, such as
holonomicity, regularity, and reducibility of monodromy
representation. We also examine its effect on solutions,
characteristic varieties, and singular loci.  
By applying $\Pi^{\tilde A}_B$ to suitable binomial $D$-modules, we shed new light on the $D$-module theoretic properties of
systems of classical hypergeometric differential equations.
\end{abstract}
\maketitle
\vspace{-4mm}
\dedicatory{\begin{center}\emph{
In memory of Mikael Passare.
}\end{center}}

\setcounter{tocdepth}{1}

In the late 1980s, Gelfand, Graev, Kapranov and Zelevinsky introduced
a new kind of hypergeometric functions and differential equations. The study of
these objects produced major advances in combinatorial algebraic
geometry, culminating in the landmark book~\cite{gkzbook}. 
A key feature of these new hypergeometric objects is the presence of a
torus action. To illustrate this, consider the Gauss
hypergeometric series 
\[
f(z) = \sum_{n=0}^\infty
\frac{a(a+1)\cdots(a+n-1)b(b+1)\cdots(b+n-1)}{c(c+1)\cdots(c+n-1)}
\frac{z^n}{n!}, 
\]
where $a,b,c \in \CC$ and $c\notin\ZZ_{\geq 0}$, which has been studied for over two centuries. 
The Gelfand--Kapranov--Zelevinsky (or
GKZ) version of this function is
$x_1^{c-1}x_2^{-a}x_3^{-b}f(x_1x_4/x_2x_3)$,
which is homogeneous with
respect to a certain action of $(\CC^*)^3$ on $\CC^4$ derived from
(the differential equation for) $f(z)$.

Intuitively, GKZ-hypergeometric functions should be viewed as torus
equivariant versions of the classical hypergeometric functions of
Euler, Gauss, Horn, Lauricella and many others; and, at the level of
functions, this intuition is correct. 
The goal of this article is to formalize such a relationship
both globally and at
the level of differential equations, in such a way that $D$-module
theoretic results about GKZ-systems can also be applied to classical hypergeometric differential equations.
In fact, while GKZ-functions involve more variables than their classical
counterparts, there are more systematic methods for studying them from the point of view of $D$-modules than
their classical versions, thanks to the underlying torus action.

The following are the systems of differential equations that central
to this article. 
For a $d\times n$ integer matrix $A = (a_{ij})\in\ZZ^{d\times n}$, 
let $I$ be an $A$-graded binomial $\CC[\del_x]$-ideal, 
generated by elements of the form 
$\del_x^u - \lambda \del_x^v$ with $Au=Av$ (and $\lambda = 0$ is allowed)
With $\barX \defeq\CC^n$, a \emph{binomial $D_\barX$-module} depends on a choice of some $\beta
\in \CC^d$ and has the form  
\[
D_\barX/( I+ \<E-\beta\>) \defeq D_\barX/( D_\barX\cdot I+ \<E-\beta\>). 
\]
Here $E_i \defeq \sum_{j=1}^n a_{ij} x_i\dx{i} 
$ for $1 \leq i \leq d$, and 
$E-\beta$ is the sequence
$E_1-\beta_1,\dots, E_d-\beta_d$ of 
\emph{Euler operators} of $A$, 
see Remark~\ref{rem:eulers}. 

For special kinds of binomial ideals, we obtain special kinds of
binomial $D_\barX$-modules. The \emph{toric ideal}
\[
I_A \defeq \< \del_x^u - \del_x^v \mid Au=Av\> \subseteq \CC[\del_x]
\]
gives rise to a binomial $D_\barX$-module called an \emph{$A$-hypergeometric
  $D_\barX$-module}. The left $D_\barX$-ideal 
  \[
  H_A(\beta) \defeq D_\barX\cdot I_A + \< E-\beta\>
  \] 
  is called an \emph{$A$-hypergeometric system}, or \emph{GKZ-system}.

If $B$ is a Gale dual of $A$ (see 
Convention~\ref{conv:BtildeACK}), then the \emph{lattice basis ideal} associated to $B$ is 
\[
I(B) \defeq \< \del_x^{w_+} -\del_x^{w_-} \mid w = w_+ -w_- \text{ is a
  column of } B \> \subseteq \CC[\del_x]
\]
The \emph{lattice basis binomial $D_\barX$-module} is
$D_\barX/ (I(B) + \<E-\beta\>)$.

The matrix $A$ induces a monomial action of $T\defeq\dtorus$ on
$\barX$ (see Definition~\ref{def:torusAction}), and thus also on the
(algebraic) differential operators on $\barX$.  On the level of
functions, the Euler operators allow the elimination of $d$ variables;
this connects GKZ- and classical hypergeometric functions.  In order to
get a quotient construction on equivariant $D$-modules, one might try
GIT or descent techniques on the underlying varieties.  The drawbacks
to the GIT approach are: the GIT quotient of $\barX$ by $T$ is
singular; there is no room for the choice of a Gale dual, natural in
the case of lattice basis ideals. On the other hand, descent is not
suitable for handling parameters, which are a major ingredient in the
study of hypergeometric systems.  We present here a richer
construction, see Section~\ref{subsec:descent}.

In Part I, we restrict $T$-equivariant 
$D_\barX$-modules to
$X\defeq\barX\minus\Var(x_1\cdots x_n)$,  
the open torus of $\barX$, and consider the 
family of toric maps from the quotient of $X$
by $T$ to the open torus $Z$ in an affine space 
$\barZ = \CC^m$, where $m\defeq n-d$. This family 
is parametrized by invertible matrices $\tilde A$ that
extend $A$ and the Gale duals $B$ of $A$.  
We construct a functor $\Pi^{\tilde A}_B$ that 
sends $T$-equivariant $D_\barX$-modules 
to $D_\barZ$-modules. 
With moderate restrictions on its source 
category, $\Pi^{\tilde A}_B$ preserves 
holonomicity, regularity, and reducibility of monodromy representation, 
as shown in Theorem~\ref{thm:transfer theorem}. 
Under an additional irreducibility assumption 
in \S\ref{sec:E-beta}, we obtain an explicit 
presentation for $\Pi^{\tilde A}_B$ and describe its impact on solutions, characteristic varieties, and singular loci. 

In Part~II, we apply $\Pi^{\tilde A}_B$ to binomial $D_\barX$-modules in order to investigate systems arising from an $n\times m$ integer matrix $B$ of full 
rank $m$ with rows 
$B_1,\dots,B_n$, along with a vector $\kappa \in \CC^n$. 
To define these systems, let $\barZ = \CC^m$ and $\eta \defeq [z_1\dz{1},\dots,z_m\dz{m}]$, and
construct the following elements of $D_\barZ$:  
\[
q_k \defeq 
\prod_{b_{ik}>0} \prod_{\ell=0}^{b_{ik}-1} 
( B_i \cdot \eta+ \kappa_i
- \ell ) \quad \text{ and } \quad
p_k \defeq 
\prod_{b_{ik}<0} \prod_{\ell=0}^{|b_{ik}|-1} 
( B_i \cdot \eta+ \kappa_i
- \ell ) .
\]
\begin{enumerate}
\item \label{def:reg-Horn}
The \emph{Horn hypergeometric system} 
associated to $B$ and 
$\kappa$ is the left $D_\barZ$-ideal
\begin{equation}
\label{eqn:horndef}
\Horn(B,\kappa) \defeq 
D_\barZ\cdot\< q_k - z_k p_k \mid k=1,\dots m\> 
\subseteq D_\barZ.
\end{equation}
\item \label{def:sHorn}
The \emph{saturated Horn hypergeometric system} 
associated to $B$ and 
$\kappa$ is the left $D_\barZ$-ideal
\begin{equation}
\label{eqn:shorndef}
\sHorn(B,\kappa) \defeq 
D_Z \cdot\< q_k - z_k p_k \mid k=1,\dots m\> 
\cap D_\barZ = D_Z \cdot \Horn(B,\kappa) \cap D_\barZ
\subseteq D_\barZ.
\end{equation}
\end{enumerate}

Once $B$ and $\kappa$ are fixed, the holomorphic solutions of these 
two types of systems coincide, while the modules themselves are truly 
different, see Example~\ref{ex:preliminaryCounterexample}. 
We show in 
Corollary~\ref{cor:Pi of lattice basis is saturated Horn} that 
$\sHorn(B,\kappa)$ is essentially captured by the image of 
$\Pi^{\tilde A}_B$, and it therefore has a natural $D$-module theoretic relationship with its equivariant counterpart, the lattice basis binomial $D_\barX$-module. 
We use this in  
Corollaries~\ref{cor:binomial quotients}.\eqref{item:binomial:Horn-regHol} and~\ref{cor:monodromy horn} to characterize 
the holonomicity, regularity, and reducibility of monodromy representation
of saturated Horn systems. 
Despite their classical nature, 
these are the first $D$-module theoretic results for Horn systems, 
aside from ~\cite{sadykov-mathscand,DMS}. 
We close with
a discussion of the 
relationship between the image
under $\Pi^{\tilde A}_B$ of an 
$A$-hypergeometric system and the Horn--Kapranov
uniformization of discriminantal varieties in~\S\ref{sec:binomial Sing}. 

\vspace{-2mm}
\subsection*{Outline for Part~I}
In \S\ref{sec:defns}, 
we provide background on $D$-modules and 
define the torus actions we 
will consider, while we 
introduce important categories of 
equivariant $D$-modules in 
\S\ref{sec:categories}. 
We work with the torus invariants functor for $D$-modules on $X = (\CC^*)^n$ in 
\S\ref{sec:invariants}. 
Over $\barX = \CC^n$ in \S\ref{sec:Pi}, 
we define and state initial properties of the functor $\Pi^{\tilde A}_B$, 
while \S\ref{sec:E-beta} 
provides more refined information. 

\vspace{-3mm}
\subsection*{Outline for Part~II} 
In \S\ref{sec:binomial}, 
we discuss binomial $D_\barX$-modules, 
applying $\Pi^{\tilde A}_B$ to them in 
\S\ref{sec:binomial hol Pi}. 
This yields explicit characterizations of $D$-module properties for saturated Horn systems in terms of their parameter sets. 
We explain the 
relationship between the image
under $\Pi^{\tilde A}_B$ of an 
$A$-hyper\-geo\-met\-ric system and the Horn--Kapranov
uniformization of discriminantal varieties in~\S\ref{sec:binomial Sing}. 
 
\vspace{-2mm}
\subsection*{Acknowledgements}
We are grateful to Frits Beukers, Alicia Dickenstein, Brent Doran,
Anton Leykin, Ezra Miller, Christopher O'Neill, Mikael Passare, and
Bernd Sturmfels, who have generously shared their insight and
expertise with us while we worked on this project.  
Parts of this work were carried out at the Institut Mittag-Leffler
program on Algebraic Geometry with a view towards Applications and the
MSRI program on Commutative Algebra. We thank the program organizers
and participants for exciting and inspiring research atmospheres.  

\section*{{\bf Part I: \ Torus invariants and \texorpdfstring{$D$}{D}-modules}}

\section{\texorpdfstring{$D$}{D}-modules and torus actions}
\label{sec:defns}

Fix integers $n > m > 0$ and set $d\defeq n-m$. After reviewing some
background on $D$-modules and their solutions, we describe the action
of the algebraic $d$-torus $T \defeq (\CC^*)^d$ on $\CC^n$ and show how
this action extends to regular functions, differential operators, and
holomorphic germs. 

\subsection{\texorpdfstring{$D$}{D}-modules}
\label{subsec:Dmods}
Let $\barX \defeq \CC^n$ with coordinates $x \defeq (x_1,\dots,x_n)$. The
\emph{Weyl algebra} $D_\barX$ is the ring of differential operators on
$\barX$, which is a quotient of the free associative $\CC$-algebra
generated by $x_1,\dots,x_n, \dx{1},\dots,\dx{n}$ by the two-sided ideal  
\[
\left\< x_i x_j - x_j x_i,\ 
\dx{i}\dx{j} - \dx{j}\dx{i}, \ 
\dx{i}x_j - x_j\dx{i} +\delta_{ij}
\mid i,j\in\{1,\dots,n\} \right\>,
\] 
where $\delta_{ij}$ is the Kronecker $\delta$-function. Let $\barZ \defeq
\CC^m$ with coordinates $z \defeq (z_1,\dots,z_m)$, and define its Weyl
algebra $D_\barZ$ analogously, with $\dz{i}$ denoting the operator for
differentiation with respect to~$z_i$. Throughout this article, let  
\begin{align*}
\theta_i \defeq x_i\dx{i}, 
\quad \eta_i \defeq z_i\dz{i}, \quad  
\theta\defeq [\theta_1\ \theta_2\ \cdots\ \theta_n], 
\quad \text{ and } \quad 
\eta \defeq [\eta_1\ \eta_2\ \cdots\ \eta_m]. 
\end{align*}
Other relevant rings are the Laurent Weyl algebras $D_X$ and $D_Z$,
defined as the rings of linear partial differential operators with
Laurent polynomial coefficients, that is,  
\[
D_X 
	= \CC[x^{\pm}]\otimes_{\CC[x]}D_{\barX} \quad \text{and} \quad 
D_Z 
	= \CC[z^{\pm}]\otimes_{\CC[z]}D_{\barZ},
\]
where $\CC[x^{\pm}]$ and $\CC[z^{\pm}]$ denote denote the rings of differential operators on 
\[
X \defeq 
\barX\minus \Var\left(x_1\cdots x_n \right) 
= (\CC^*)^n \quad
\text{and} \quad
Z \defeq 
\barZ\minus \Var\left(z_1\cdots z_m \right) 
= (\CC^*)^m.
\]

The weight filtration induced by
$F \defeq (\boldzero_n,\boldone_n) \in \QQ^{2n}$ is called
the \emph{order filtration} on $D_{\barX}$. For $k\in\QQ$, its $k$th filtered piece
\[
F^k D_\barX \defeq \CC\cdot \{x^u\del_x^v \mid  F\cdot(u,v) = |v| \leq k\} 
\]
consists of   
all operators of order at most $k$. 
For any $P$ of order $k$, we denote its \emph{symbol} by   
$\ini_F(P) \defeq P+F^{<k}D_\barX\in\gr^{F,k} D_\barX \defeq F^k D_\barX /F^{<k} D_\barX$. 
By our convention, the associated graded ring $\gr^F D_\barX \defeq \bigoplus_k \gr^{F,k}D_\barX$ is isomorphic to the coordinate ring of $T^*\barX\cong \CC^{2n}$. 
Denoting $\ini_F(\dx{i})=\xi_i$ and abusing language by
writing $x_i$ for $\ini_F(x_i)$, we have $\gr^F D_\barX\cong
\CC[x_1,\dots,x_n,\xi_1,\dots,\xi_n]$. Similarly, $\gr^F
D_\barZ \cong \CC[z_1,\dots,z_m,\zeta_1,\dots,\zeta_m]$,
where $\zeta_i$ is the symbol of $\dz{i}$. 

Let $M$ be a finitely generated left $D_\barX$-module. The weight vector
$F = (\boldzero_n,\boldone_n)$ on $D_\barX$ induces the \emph{order filtration} on (a presentation of) $M$. The
\emph{characteristic variety} of $M$ is $\supp(\gr^F(M))\subseteq T^*\barX
\cong\CC^{2n}$. We denote this set by $\charVar(M)$; it is well-defined (see~\cite{schulze-walther-duke} for references), Zariski closed, and agrees with the zero set of the ideal $\ann(\gr^F(M))$. 
Bernstein's inequality says
that the characteristic variety of a left $D_\barX$-ideal has dimension at least 
$n$~\cite{bernstein}. Smith~\cite{smith-characteristic-variety}
refined this result, showing that every component of $\charVar(M)$
has dimension at least $n$.   
We say that $M$ is
\emph{holonomic} if it is $0$ or its characteristic variety has dimension $n$.

The \emph{rank} of the module $M$ is 
$\rank(M) \defeq \dim_{\CC(x)} ( \CC(x)\otimes_{\CC[x]} M)$. 
If $M$ is holonomic, then $\rank(M)$ is finite, see~\cite[Proposition~1.4.9]{SST}. 

The projection of $\charVar(M) \minus \Var(\xi_1,\ldots, \xi_n )$ 
onto the $x$-coordinates is called the \emph{singular locus of $M$},
denoted $\Sing{M}$. Any $p\in\Sing{M}$ is called a \emph{singular
  point} of $M$. 

Let $Y$ be a complex manifold of dimension $n$ with a point $p\in Y$.
The holonomic left $D_Y$-module $M$ is said to be \emph{regular 
at $p$} if for any curve $C\subseteq Y$ passing through $p$, with smooth 
locus $\iota\colon C_s\into Y$, there is an open neighborhood $U$ of $p$ in $Y$ such 
that all the sheaves $\mathbb{L}^k \iota^*(M|_U)$ are connections on 
$C_s$ with regular singular ODEs on a smooth compactification $\overline{C}$ of $C_s$. 
(These sheaves are zero outside the range $-n+1\le k\le 0$.)   
The module $M$ is \emph{regular holonomic} if it is regular holonomic
at every point $p\in\barX$. 

The notion of regularity is equivalent to requiring that the natural
restriction map from formal to analytic solutions of $M$ be an
isomorphism in the derived category. As such, it generalizes the
classical definition of regular (Fuchsian) singularities for an
ordinary differential equation. Note that, since we consider a
compactification of $C$, regular holonomicity on $Y$ includes
information about the behavior of (derived) solutions at infinity.  

The categories of holonomic and regular holonomic $D_Y$-modules are
Abelian and closed under the formation of extensions, and they
form full subcategories of the category of $D$-modules.  Both
holonomicity and regularity are preserved by direct and inverse 
images along morphisms of smooth varieties. See~\cite{borel} for more
details. 

A holonomic left $D_Y$-module $M$ is said to have 
\emph{irreducible monodromy representation}
if the module $M(y)\defeq \CC(Y)\otimes_{\CC[Y]}M$ is an irreducible
module over  $D(y)\defeq \CC(Y)\otimes_{\CC[Y]}D_Y$. Here, $\CC[Y]$ and
$\CC(Y)$ are the rings of regular and rational functions on $Y$, 
respectively. If $M$ does not have irreducible monodromy 
representation, then its solution space has a  nontrivial proper
subspace that is monodromy invariant, since the fundamental group of
$Y\minus\Sing{M}$ is finitely presented.

Throughout, on any $\CC$-scheme, we mean by ``differential operators''
the sheaf of $\CC$-linear differential operators on the corresponding
structure sheaf.  
Since we consider only products of affine spaces and tori as
underlying manifolds, which are all $D$-affine, we may restrict our
attention to global sections. 

\subsection{Solutions of \texorpdfstring{$D$}{D}-modules}
\label{subsec:solutions}

Let $\mathcal{O}_{p, \barX}^{{\rm an}}$ be the space of germs 
of holomorphic functions at $p \in X$, and use $\bullet$ for the action of
$D_\barX$ on $\OO_{p,\barX}^{\text{an}}$. Given a left $D_\barX$-ideal $J$, the 
\emph{solution space $\Sol_p(D_\barX/J)$ at $p$} consists of elements of
$\mathcal{O}_{p, \barX}^{{\rm an}}$ that are annihilated by $J$:
\[
\begin{array}{rcl}
\hom_{D_\barX}(D_\barX/J, 
	\OO_{p,\barX}^{\text{an}})
 & \stackrel{\sim}{\longrightarrow}  & 
 \{ f\in\OO_{p,\barX}^{\text{an}} 
 	\mid P\bullet f = 0 \ \forall P\in I \}
 = \Sol_p(D_\barX/J).\\
\Phi & \mapsto & f(x) \defeq \Phi(1+J)
\end{array}
\]

\begin{theorem}[Kashiwara, see~{\cite[Theorem~1.4.19]{SST}}]
\label{thm:CKK}
Let $J$ be a left $D_\barX$-ideal such that $D_\barX/J$ is holonomic, and let $p$ be a 
nonsingular point of $D_\barX/J$. Then $\dim_{\CC} (\Sol_p(D_\barX/J))$ is finite, 
independent of $p$, and equal to $\rank(D_\barX/J)$. 
\qed
\end{theorem}

Given $w\in\RR^n$, $(-w,w)\in\RR^{2n}$ induces a filtration on $D_\barX$.  
Note that $\gr^{(-w,w)}D_\barX \cong D_\barX$, which is in contrast to
the case in \S\ref{subsec:Dmods} for a weight vector $L\in\RR^{2n}$;
however, for $P\in D_\barX$, we may analogously define
$\ini_{(-w,w)}(P)\in D_\barX$.  
If $J\subseteq D_\barX$ is a left $D_\barX$-ideal, then set
$\gr^{(-w,w)}(J) \defeq \< \ini_{(-w,w)}(P)\mid P\in J\>$, which by abuse
of notation, we view as a left $D_\barX$-ideal.  

\begin{definition}
\label{def:generic:weight}
Let $J \subseteq D_\barX$ be a left $D_\barX$-ideal. We say that
$w\in\RR^n$ is a  \emph{generic weight vector} for $D_\barX/J$ if
there exists an ($n$-dimensional) open rational polyhedral cone  
$\Sigma\subseteq\RR^n_{>0}$ with trivial lineality space such that
$w\in\Sigma$ and for all $w'\in\Sigma$, $\gr^{(-w,w)}(J) =
\gr^{(-w',w')}(J)$.  
\end{definition}

\begin{definition}
\label{def:basic nilsson}
Let $J \subseteq D_\barX$ be a left $D_\barX$-ideal, and assume that
$w\in\RR^n$ is a generic weight vector for $D_\barX/J$.  
Write $\log(x) \defeq (\log(x_1),\dots,\log(x_n))$. 
A formal solution $\phi$ of $D_\barX/J$ is called a \emph{basic
  Nilsson solution of $D_\barX/J$ in the direction of $w$} if it has
the form  
\[
\phi(x) = \textstyle\sum_{u\in C} x^{v+u}p_u(\log(x)),
\]
for some vector $v\in\CC^n$, such that the following conditions are satisfied:
\vspace{-2mm}
\begin{enumerate}
\item 
$C$ is contained in 
$\Sigma^*\cap\ZZ^n$, where $\Sigma$ is as in Definition~\ref{def:generic:weight}. 
Here, $\Sigma^*$ is the \emph{dual cone} of $\Sigma$, which consists
of the vectors $u\in\RR^n$ with $u\cdot w'\geq 0$ for all
$w'\in\Sigma$.  
\item 
the $p_u$ are polynomials, and there exists $k\in\ZZ$ such that
$\deg(p_u)\leq k$ for all $u\in C$. 
\item 
$p_0\neq 0$. 
\end{enumerate}
\vspace{-2mm}
The set $\supp(\phi) \defeq \{ u\in C\mid p_u\neq 0\}$ is called the \emph{support} of $\phi$. 
The $\CC$-span of the basic Nilsson solutions of $D_\barX/J$ in the
direction of $w$ is called the \emph{space of formal Nilsson solutions
  of $D_\barX/J$ in the direction of $w$}, denoted $\sN_w(D_\barX/J)$.  
\end{definition}

\begin{theorem}\cite[Theorems~2.5.1 and~2.5.14]{SST}
\label{thm:solutions-of-regular-holonomic}
Let $J$ be a left $D_\barX$-ideal such that $D_\barX/J$ is regular holonomic.
If $w\in \RR^n$ is a generic weight vector for $D_\barX/J$, then 
\[
\dim_\CC(\sN_w(D_\barX/J)) = \rank(D_\barX/J). 
\]
Further, there is an open set $U\subseteq\barX\minus\Sing{D_\barX/J}$
such that the basic Nilsson solutions of $D_\barX/J$ in the direction
of $w$ simultaneously converge at each $p\in U$ and form a basis for
$\Sol_p(D_\barX/J)$.  
\qed
\end{theorem}

\subsection{Torus actions in the Weyl algebra}
\label{subsec:torus actions}

\begin{convention}
\label{conv:A}
Fix a $d\times n$ integer matrix $A$ of rank $d$ whose columns span
$\ZZ^d$ as a lattice. We denote the columns of $A$ by $a_1,\dots,a_n$. 
Throughout this article we assume that $A$ is \emph{pointed};  
in other words, there exists $h \in \RR^d$ such that 
$h\cdot a_i > 0$ for all $i=1,\dots,n$. 
\endrk
\end{convention}

\begin{definition}
\label{def:torusAction}
We consider the action of the torus $T=(\CC^*)^d$ on $\barX$ given by 
\[
(t_1,\dots,t_d) \diamond (x_1,\dots,x_n) \defeq
(t^{a_1}x_1,\dots,t^{a_n}x_n).
\] 
\end{definition}
We let $\sT$ denote the torus orbit of $\boldone_n$ in $\barX$. By our
assumptions on $A$, the map $T\to\sT$ given by $t\mapsto t\diamond \boldone_n$ is
an isomorphism of varieties.
This torus action induces an action of $T$ on $D_\barX$ via
\begin{equation}
\label{eqn:action}
t \diamond x_i \defeq t^{a_i} x_i, \quad 
t \diamond \dx{i} \defeq t^{-a_i} \dx{i}, 
\hspace{.5cm} 
\text{ for } 1 \leq i \leq n \text{ and } 
t=(t_1,\dots,t_d) \in T. 
\end{equation}
An element $P \in D_\barX$ is \emph{torus homogeneous of weight $a \in
  \ZZ^d$} if $t \diamond P = t^a P$ for all $t \in T$. A left
$D_\barX$-ideal is torus equivariant if and only if it is generated by
torus homogeneous elements.

The torus action on $D_\barX$ imposes a multigrading, called the
\emph{$A$-grading}, given by 
\[
\deg(x_i)\defeq a_i
\quad\text{and}\quad 
\deg(\dx{i})\defeq -a_i.
\] 
The grading group is $\ZZ A=\ZZ^d$, the Abelian group generated by the 
columns of $A$.
An element $P = \sum \lambda_{u,v}x^u\del_x^v \in D_\barX$ 
has degree $a \in \ZZ A = \ZZ^d$ if and only if 
$Au-Av=a$ whenever $\lambda_{u,v} \neq 0$. 
The term \emph{$A$-graded} is used to emphasize that the torus
action is defined by $A$. 

The torus action also passes to germs of functions on $\barX$. 
For $t \in T$ and $x \in \barX$, denote 
$t \diamond x \defeq (t^{a_1}x_1,\dots,t^{a_n}x_n)$.
A function that satisfies
$\varphi(t \diamond x) = t^{\beta}\varphi(x_1,\dots,x_n)$ 
for all $t$ in an open subset of $T$ has degree $\beta\in\CC^d$. 
Note that any $\beta \in \CC^d$ may occur as a weight of a function. 
The subspace of $\OO_{p, \barX}^{{\rm an}}$ consisting of $A$-graded
germs of degree $\beta$ is denoted $\big[\OO_{p, \barX}^{\text{an}}\big]_{\beta}$.

\begin{remark}
\label{rem:eulers}
The space 
$\big[\OO_{p, \barX}^{\text{an}}\big]_{\beta}$ 
can be viewed as the set of 
$\beta$-eigenvectors for the 
Euler operators $E$ of $A$. 
As $\beta$ defines a character of
$T$, the action of $T$ on $X$ (really, on $T^*X$) can be viewed as
an action of its Lie algebra. 
Loosely, the operators in $E$ are the (Fourier transforms of the) 
pushforwards of generators of the Lie algebra of the torus. Thus, the
solutions of the Euler operators $E-\beta$ are precisely the functions
which are infinitesimally torus homogeneous of weight $\beta$.  
In other words, 
\begin{align*}
\hspace{5cm}
{\rm Hom}(D_\barX/ \< E-\beta\>, \ \OO_{p, \barX}^{{\rm an}}) 
	\cong \big[\OO_{p, \barX}^{{\rm an}} \big]_{\beta}.
\hspace{4.54cm}
\hexagon
\end{align*}
\end{remark}

\begin{remark}
\label{rem:Agrd nilsson}
If $J$ is an $A$-graded left $D_\barX$-ideal, then any series solution
of $D_\barX/J$ may be decomposed as a sum of $A$-homogeneous series,
each of which is also a solution of $D_\barX/J$. In particular, for
a generic weight vector $w$, $\sN_w(D_\barX/J)$ of Definition~\ref{def:basic nilsson} is spanned by basic Nilsson series 
$\phi(x) = \sum_{u\in C} x^{v+u}p_u(\log(x))$ whose support $C$ is
contained in $\Sigma^*\cap\ker_{\ZZ}(A)$. Further, if
$\<E-\beta\>\subseteq J$, then $Av = \beta$.  
\endrk
\end{remark}

\section{Categories of \texorpdfstring{$A$}{A}-graded \texorpdfstring{$D$}{D}-modules}
\label{sec:categories}

In this section, $Y$ is either $X$ or $\barX$, and $\Dmods{Y}$ is the
category of finitely generated $D_Y$-modules. We introduce certain
subcategories of $\Dmods{Y}$ that will be used in this article, and
prove some of their basic properties.

\begin{definition}
\label{def:categories}
Let $\Amods{Y}$ denote the category of finitely generated $A$-graded
$D_{Y}$-modules with $A$-graded morphisms of $A$-degree zero. We let
$\Ahol{Y}$ denote the full subcategory of $\Amods{Y}$ given by
$A$-graded holonomic $D_Y$-modules.   
\end{definition}

\begin{remark}
\label{rem:TsT iso} 
Recall that $\sT=T\diamond\boldone_n \cong T$.  
The ring of differential operators on $\sT$ is the subgroup $D_\sT$ of 
$D_X/(\<x^u-1 \mid u \in \ZZ^n, Au=0\>\cdot D_X)$ generated by the 
monomials $x^v$ with $v\in\ZZ^n$ and the operators $E_1,\dots,E_d$.  
This is indeed a ring and is isomorphic to $D_T$ via $x^v \mapsto t^{Av}$
for $v \in \ZZ^n$ and $E_i \mapsto t_i\del_{t_i}$ for $i=1,\dots,d$.  
Note that $[D_\sT]^T = \CC[E_1,\dots,E_d] =: \CC[E]$ is (isomorphic 
to) a polynomial ring in $d$ variables generated by the 
$T$-equivariant vector fields on $X$.  The containment $D_\barX\subseteq D_X$ is 
such that we may endow $D_\barX$ with a $[D_\sT]^T$-action using the
$E_i$. We shall use this lifted action in the sequel without further mention.
\endrk
\end{remark}

\begin{definition}
\label{def:Thol}
We denote by $\AEmods{Y}$ the smallest full subcategory of $\Amods{Y}$ 
that contains each module $M$ with the property that, for each
homogeneous element $\gamma \in M$, the $D_\sT$-module 
$D_\sT/(D_\sT\cdot \ann_{[D_\sT]^T}(\gamma))$ 
is holonomic.
\end{definition}

Note that there is a natural functor from $\AEmods{\barX}$ to
$\AEmods{X}$ given by  restriction to $X$. 
Also, note that the categories $\Amods{Y}$, $\Ahol{Y}$, and $\AEmods{Y}$ are
all Abelian, and $\Amods{Y}$ and $\Ahol{Y}$ are clearly closed under extensions.

\begin{lemma}
\label{lem:Tholclosed}
The category $\AEmods{Y}$ is closed under extensions.
\end{lemma}

\begin{proof}
Let $0\to M\to N\to P\to 0$ be a short exact sequence in $\Amods{Y}$
with $M$ and $P$ in $\AEmods{Y}$. For any homogeneous $\gamma\in N$, let
$J\defeq D_\sT\cdot \ann_{[D_\sT]^T}(\gamma +M)$, the annihilator of
$\gamma+M\in P$, be generated by $P_1,\ldots,P_k\in[D_\sT]^T$. Then
$P_1\cdot \gamma,\ldots,P_k\cdot \gamma$ all lie in $M$ and are
hence annihilated by ideals $I_1,\ldots,I_k$ with generators in
$[D_\sT]^T$.  Therefore 
\[
D_\sT \cdot
\ann_{[D_\sT]^T}(\gamma)\supseteq \left(\textstyle\bigcap_j 
I_j\cdot P_1\right)+\ldots+\left(\textstyle\bigcap_j I_j\cdot P_k\right)
\supseteq\left(\textstyle\bigcap_j I_j\right)\cdot J. 
\]
Note that if $J_1$ and $J_2$ are left $D_\sT$-ideals such that each $D_\sT/J_i$ is holonomic,
then $D_\sT/(J_1\cap J_2)$ is also holonomic because there is an exact sequence
\[
0\ \to \ \frac{D_\sT}{J_1\cap J_2} 
\ \to \ \frac{D_\sT}{J_1}\oplus\frac{D_\sT}{J_2} 
\ \to\  \frac{D_\sT}{J_1+J_2}\ \to\  0.
\]
Since each $D_\sT/I_j$ is holonomic by assumption, so is
$D_\sT/\bigcap_j I_j$.  Now if $J_1, J_2$ are such that $D_\sT/J_i$ are
holonomic for each $i$, then $D_\sT/(J_1J_2)$ is also holonomic.  
Indeed, the exact sequence
\[
0\ \to \ \frac{J_2}{J_1J_2}
\ \to \ \frac{D_\sT}{J_1J_2}
\ \to \ \frac{D_\sT}{J_2}
\ \to \ 0
\]
shows that it suffices to prove holonomicity for $J_2/(J_1J_2)$, which
is the quotient of the holonomic module $(D_\sT/J_1)^k$ under the map 
induced by any surjection $D_\sT^k\onto J_2$.  Thus $D_\sT/(\bigcap_j I_j\cdot J)$ is
holonomic, so the same is true for $D_\sT/D_\sT\cdot \ann_{[D_\sT]^T}(\gamma)$, as desired.
\end{proof}

\begin{theorem}
\label{thm:Ahol sub}
The category $\Ahol{Y}$ is a subcategory of $\AEmods{Y}$. 
\end{theorem}
\begin{proof}
We begin by proving the result when the matrix $A$ consists
of the top $d$ rows of an $n\times n$ identity matrix and $Y=\barX$.
In this case, the torus $\sT$ is simply 
$\Var(x_{d+1}-1,\dots,x_n-1)\subseteq X$. 
Thus $D_{\sT}$ is the $\CC[x_1^{\pm},\ldots,x_d^{\pm}]$-algebra generated by
$\theta_1,\ldots,\theta_d$, while the invariant ring $[D_{\sT}]^T$ is
$\CC[\theta_1,\ldots,\theta_d]$. 
If $M$ is a holonomic $A$-graded $D_\barX$-module and $\gamma\in M$ is
a fixed homogeneous element, then the annihilator of $\gamma$ is a
holonomic $A$-graded ideal, say $I\defeq \ann_{D_\barX}(\gamma)$. As such, the
$b$-function $b_{\gamma,i}(s)$ for restriction to $x_i=0$, $1\le i\le
d$, exists (i.e., is nonzero) and satisfies
\begin{align}\label{eq:bfunc}
b_{\gamma,i}(\theta_i)\in I+V^{-1}_i D_\barX.
\end{align}
Here $V^\bullet_iD_\barX$ denotes the Kashiwara--Malgrange filtration
on $D_\barX$, which is given by
\[
V^k_i D_\barX\defeq \Span_\CC \{x^u\del^v\in D_\barX
\mid v_i-u_i\le k\}.
\]
(A discussion of $b$-functions for restriction can be found in~\cite[\S\S5.1-5.2]{SST}.)
Taking the $A$-degree zero part of any equation expressing the
containment~\eqref{eq:bfunc} shows that  $b_{\gamma,i}(\theta_i)\in I$
because no element of $V^{-1}_i D_\barX$ has $A$-degree $0$. (Recall
that, for the moment, $A$ consists of the top $d$ rows of an $n\times n$ identity
matrix.) In other words, for all $i$, the $D_\barX$-annihilator 
of $\gamma$ contains a polynomial in $E_i$. Now since
$\CC[\theta_i]/\<b_{\gamma,i}(\theta_i)\>$ is Artinian, so is the
quotient $\CC[\theta_1,\ldots,\theta_d]/\ann_{[D_\sT]^T}(\gamma)$ of 
$\CC[\theta_1,\ldots,\theta_d]/\sum_i\<b_{\gamma,i}(\theta_i)\>$.
Theorem~\ref{thm:DsT} below completes this special case.

We next observe that the case $Y=X$ follows from the case $Y=\bar X$
(for any fixed $A$). Indeed, if $M$ is holonomic and $A$-graded on $X$, 
then the pushforward of $M$ to $\barX$ (which is just $M$, viewed as a 
$D_\barX$-module) has the same properties and is therefore in 
$\AEmods\barX\supseteq \AEmods X$.

We next consider $Y=X$ but general $A$. The theory of elementary
divisors ascertains the existence of matrices $P\in \Gl(d,\ZZ)$ and
$Q\in \Gl(n,\ZZ)$ such that $PAQ$ is the top $d$ rows of a diagonal
$n\times n$ matrix. Substituting $PA$ for $A$ is just a change of 
coordinates on the grading group and the Euler operators.  On the 
other hand, replacing $A$ by $AQ$ corresponds to the change on the 
grading group and the Euler operators induced by the monomial change 
of coordinates $x\mapsto x^Q$ on $X$.  This change is $T$-invariant on
the category of $D_X$-modules, and thus also on $\Ahol{X}$ and
$\AEmodules(X)$. Now $\ZZ A = \ZZ^d$ implies that $PAQ$ can be chosen
to equal the first $d$ rows of an $n\times n$ identity matrix, which
settles the case that $A$ is arbitrary and $Y=X$.

Finally, we reduce the case $Y=\barX$ to the case $Y=X$. To that end,
use induction on $n$. Let $i\colon X\into \barX$ be the natural
embedding and fix a module $M$ in $\Ahol \barX$.  For $n=0$, the
result is trivial.  For $n=1$, consider the exact sequence of
local cohomology 
\[
0 \to H^0_{x_1}(M) \to M\to i_*i^*M \to H^1_{x_1}(M) \to 0.
\]
The outer terms are supported in $x_1=0$. By Kashiwara equivalence, the case
$n=0$, and since $\AEmods \barX$ is closed under extensions, the proposition
holds for $M$ precisely if it holds for $i_*i^*M$. But we already proved
it holds for $i^*M$, so the case $n=1$ is proven.

The general case is slightly more involved, but it suffices to
indicate the case $n=2$. Given $i_1\colon
X_1=\barX\minus\Var(x_1)\into \bar X$ and $i_2\colon
X_2=X_1\minus\Var(x_2)\into X_1$, there are two exact sequences: 
\begin{align*}
0 \to H^0_{x_1}(M)\to M \to {i_1}_*{i_1}^*M \to H^1_{x_1}(M)\to 0
\qquad \text{and}\qquad 
\\
0 \to H^0_{x_2}({i_1}_*{i_1}^*M) \to {i_1}_*{i_1}^*M \to
{i_2}_*{i_2}^*{i_1}_*{i_1}^*M \to H^1_{x_2}({i_1}_*{i_1}^*M) \to 0.
\end{align*}
By Kashiwara equivalence and the case $n=1$, the outer terms in both
sequences satisfy the proposition. As $\AEmods\barX$ is closed under
extensions, the proposition holds for $M$ if and only if it holds for
${i_1}_*{i_1}^*M$, and that happens if and only if it holds for
${i_2}_*{i_2}^*{i_1}_*{i_1}^*M$. However, the latter module is the
pushforward to $\barX$ of the restriction of $M$ to $X$, so the case
$n=2$ follows from the case $Y=X$ above.
\end{proof}

\begin{example}
The module $D_{\bar X}/(D_{\bar X}\cdot E)$ is in $\AEmods{\bar X}$ but
does not belong to $\Ahol{\bar X}$. 
\endrk
\end{example}

The following result can be used to give alternative descriptions for
the objects of $\AEmods{Y}$, where $Y$ is still $X$ or $\barX$. The
equivalence of the second and last items in Theorem~\ref{thm:DsT}
was proven in~\cite[Proposition~2.3.6]{SST} for cyclic modules, using
different methods. 

\begin{theorem}
\label{thm:DsT}
The following are equivalent for a finitely generated $A$-graded $D_\sT$-module $M$:
\begin{enumerate}
\vspace{-2mm}
\item $M$ is regular holonomic.
	\label{item:reg hol}
\item $M$ is holonomic. 
	\label{item:hol}
	\label{item:Lhol all}
\item $\CC[E]/\ann_{\CC[E]}(\gamma)$ is 
	Artinian for all $\gamma\in M$. 
	\label{itcomplem:Artinian}
\end{enumerate}
\end{theorem}
\begin{proof}
The isomorphism between $D_\sT$ and $D_T$ in Remark~\ref{rem:TsT iso} 
induces an equivalence of categories between $\Amods{\sT}$ and the 
category of finitely generated left $D_T$-modules that are equivariant with respect 
to the action induced by $T$ acting on itself by multiplication. 
Under this equivalence of categories, (regular) holonomic modules correspond to 
(regular) holonomic modules. 
Given these considerations, it is enough to prove the statements in the context of
$D_T$-modules.

First note that the module $D_T/(\<t_i\del_{t_i}-\beta_i \mid
i=1,\dots,d\>\cdot D_T)$ is regular holonomic because these are Fuchsian equations in
disjoint variables. 
This implies that if $I$ is an equivariant left $D_T$-ideal whose degree zero part
$I_0$ is an Artinian ideal in the polynomial ring
$\CC[t_1\del_{t_1},\dots,t_d\del_{t_d}]$, then the module $D_T/I$ is
regular holonomic, since the category involved is Abelian and closed under extensions. This shows that \eqref{itcomplem:Artinian} implies \eqref{item:reg hol} and \eqref{item:hol} 
in the cyclic case. 

For the reverse implications, still in the cyclic case, suppose that
$I$ is an equivariant left $D_T$-ideal. If $\CC[t\del_t]/I_0 =
\CC[t\del_t]/\ann_{\CC[t\del_t]}(1)$ 
is not Artinian, then $D_T/I$ has infinite
rank, and therefore cannot be (regular) holonomic. 

Finally, we consider the general case. Let $M$ be a finitely generated
$A$-graded equivariant $D_\sT$-module. Since the category of (regular) holonomic modules is closed under
extensions, it is not hard to show that $M$ is (regular) holonomic if and only if it has a finite
composition chain such that each composition factor is cyclic and
(regular) holonomic. Note that the length of such a maximal composition chain depends only on $M$ (and is called the \emph{holonomic length} of $M$).  

Let $\gamma\in M$, and let $N$ be the $\sT$-submodule of $M$ generated
by $\gamma$. The $\CC[E]$-annihilator of $\gamma$ as an element of $M$
is equal to the $\CC[E]$-annihilator of $\gamma$ as an element of
$N$. Since $N$ is cyclic, we have already shown that $N$ satisfies
\eqref{item:reg hol} or \eqref{item:hol} precisely if it satisfies \eqref{itcomplem:Artinian}.  But $M/N$ has smaller
holonomic length than $M$, so the result follows by induction.
\end{proof}

\section{Torus invariants of \texorpdfstring{$D_X$}{DX}-modules}
\label{sec:invariants}

In this section, we consider the exact functor $[-]^T\colon M\mapsto
M_0$ on $\Amods{X}$, which selects the $A$-degree zero part (that is,
the torus invariant part) of a module $M$.  The natural target of this
functor is the category of $[D_X]^T$-modules. Viewing $X$ as the
product $(X/\sT)\times\sT$, $M_0$ inherits an action by $D_{X/\sT}$. Any
\'etale morphism $X/\sT\to Z$ induces (by $D$-affinity) an action of
$D_Z$ on $D_{X/\sT}$ and hence on the category of
$D_{X/\sT}$-modules. This section explains how we may consider the category of $D_Z$-modules to be the target of the invariant functor $[-]^T$. 

\subsection{The functor \texorpdfstring{$\Delta^{\tilde A}_B$}{Delta^{tilde A}_B}}

For our purposes, the important cases are those for which the
composition $X\to X/\sT\to Z$ is a monomial map. These situations are
parameterized by two pieces of data: the possible splittings of $X$
and the Gale duals of $A$. We encode this information in two matrices,
$\tilde A$ and $B$, and denote our  invariants functor
$[-]^T$ by $\Delta^{\tilde A}_B$, as
defined in~\eqref{eq:def Delta}. We construct this
functor in Definition~\ref{def:Delta} and then show in Theorem~\ref{thm:invariants-are-good} that, upon restricting its source to $\AEmods{X}$, it
behaves well with respect to several $D$-module theoretic properties. 

\begin{notation}
\label{not:mu}
If $G$ is an integer $p\times q$ matrix then we denote by $\mu_G$ the
monomial morphism from $(\CC^*)^p$ to $(\CC^*)^q$ that sends
$v\in(\CC^*)^p$ to $v^G=(v^{g_1},\ldots,v^{g_q})$, where
$g_1,\ldots,g_q$ are the columns of $G$. We also denote the
corresponding morphism on the structure sheaves by $\mu_G$. For
example, $\mu_A\colon T\to X$ is the map with image $\sT$.  
\endrk
\end{notation}

Recall that $m=n-d$ and $\ZZ A=\ZZ^d$. The splittings of $X$ that factor
through monomial maps from $X$ to $(X/\sT)\times T$ are in bijection
with those matrices $\tilde A\in\Gl(n,\ZZ)$ whose top $d$ rows agree
with $A$. We denote the bottom $n-d$ rows of such $\tilde A$ by
$A^\perp$. We let $C^\perp$ and $C$ respectively denote the $n\times
d$ and $n\times m$ matrices that form the left and right parts of the
inverse matrix $\tilde C$ to $\tilde A$. 
Note that $\mu_C\colon X\to X/\sT$ comes from a splitting of $X$. 

\begin{convention}
\label{conv:BtildeACK}
An $n\times m$ integer matrix $B$ is a \emph{Gale dual} of $A$ if the
columns of $B$ span $\ker_\QQ(A)$. For the remainder of this article,
fix a Gale dual $B$ of $A$ and a matrix $\tilde A\in\Gl(n,\ZZ)$ whose
top $d$ rows agree with $A$. Since $\tilde A\tilde C$ is the identity,
$AC=0$. Thus the equation $AB=0$ implies the existence of a full rank
integer $m\times m$ matrix $K=A^\perp B$ such that $B=CK$, inducing $\mu_K\colon
X/\sT\to Z$ via $y\mapsto y^K$ as in Notation~\ref{not:mu}. Call the
composition $\mu_K\circ\mu_C\colon X\to X/\sT\to Z$ the \emph{split Gale morphism} attached
to $(\tilde A,B)$.
\endrk
\end{convention}

\begin{notation}
\label{not:Upsilon}
Suppose $(\CC^*)^q$ has a monomial action on $(\CC^*)^p$ given by a matrix $H$, 
and let $\sQ \defeq (\CC^*)^q\diamond \boldone_p\subseteq (\CC^*)^p$. 
Let $\tilde H\in\Gl(p,\QQ)$ be such that its top $q$ rows agree with $H$, 
and denote its bottom $p-q$ rows by $H^\perp$.  
Let $\tilde G = (\tilde H)^{-1}$, and let $G$ denote the matrix given by the final $p-q$ columns of $\tilde G$. 
We assume that $\tilde H$ is such that $G$ is an integer matrix. 
Using $v$ and $y$ for the respective coordinates of 
$(\CC^*)^p$ and $(\CC^*)^{p-q} \cong (\CC^*)^p/\sQ$, set 
\[
\chi_i \defeq v_i\del_{v_i}, 
\quad
\chi \defeq [\chi_1\ \chi_2\ \cdots\ \chi_p], 
\quad
\lambda_i \defeq y_i\del_{y_i}, 
\quad \text{and} \quad 
\lambda \defeq [\lambda_1\ \lambda_2\ \cdots\ \lambda_{p-q}]. 
\]
In accordance with Notation~\ref{not:mu}, 
the matrix $G$ 
induces a quotient map $\mu_G\colon
(\CC^*)^p\to(\CC^*)^{p-q} \cong (\CC^*)^p/\sQ$, and via $\mu_G$, $y$ acts as $v^G$ and
$\lambda$ acts as $H^\perp\chi$, both of which lie in
$[D_{(\CC^*)^p}]^{(\CC^*)^q}$. From this we obtain an action of
$D_{(\CC^*)^{p-q}}= \CC[y^{\pm 1}]\<\lambda\>$ on
$[D_{(\CC^*)^p}]^{(\CC^*)^q}$. We denote by $\Upsilon_G^{\tilde H}$ the
process of endowing a $[D_{(\CC^*)^p}]^{(\CC^*)^q}$-module with a $D_{(\CC^*)^q}$-module
structure via $\mu_G$.
\endrk  
\end{notation}

We will apply Notation~\ref{not:Upsilon} in two cases. The first is
for $\tilde A$ and $\mu_C\colon X\to X/\sT$, and the second is for $K^{-1}$ and $\mu_K\colon X/\sT\to Z$, where $X/\sT$ has a trivial torus action. 
Note that the functors $\Upsilon_C^{\tilde A}$ and $\Upsilon_K^{K^{-1}}$ are related to the maps in the split Gale morphism $X\to X/\sT\to Z$ attached to $(\tilde A,B)$.  

\begin{definition}
\label{def:Delta}
Using Convention~\ref{conv:BtildeACK} and Notation~\ref{not:Upsilon}, define the functor
\begin{align}
\label{eq:def Delta}
\Delta^{\tilde A}_B 
\,\, \defeq \, 
\Upsilon_K^{K^{-1}} \circ \Upsilon_C^{\tilde A} \circ [-]^T
\,\colon 
\,\, 
M\mapsto M_0,
\end{align}
which sends a finitely generated $D_X$-module $M$ to the
$[D_X]^T$-module $M_0$, viewed as $D_Z$-module.  
\end{definition}

\begin{example}
\label{ex:want fg output}
Suppose that $\tilde A = \tilde C$ is an identity matrix and $B =
C$. In this situation, we may write $X = T \times Z$, and $T$ acts on
$X$ by 
$t \diamond (s,z) =
(ts,z)$. Even in this case, it becomes apparent why, if we desire nice output, we must
restrict $\Delta^{\tilde A}_B$ to $\AEmods{X}$. To wit, the torus
invariants of $D_{T\times Z}$  are generated by  $D_Z$ and
$\CC[t_1\del_{t_1},\dots,t_d\del_{t_d}]$. In particular, while
$[D_{T\times Z}]^T = D_Z[t\del_t]$ is a $D_Z$-module, it is not finitely generated. 
Fortunately, if 
$M = \bigoplus_{\alpha\in \ZZ A} M_\alpha 
=\bigoplus_{\alpha\in\ZZ A} M_0\cdot t^\alpha$ 
is in $\Amods{T\times Z}$, then $[M]^T = M_0$ is a finitely generated
$D_Z[t\del_t]$-module. In order for $M_0$ to be a finitely generated
$D_Z$-module, more assumptions are needed to govern the impact
of $t\del_t$; the holonomicity requirement in the definition of
$\AEmods{X}$ provides precisely that. 
\endrk
\end{example}

We are now prepared to state the main result of this section, which 
concerns the restriction of 
$\Delta^{\tilde A}_B$ to $\AEmods{X}$ (see Definition~\ref{def:Thol}). 
This restriction guarantees that the output of the functor is a finitely generated 
$D_Z$-module and compatible with a number of $D$-module theoretic 
properties. 
Recall from \S\ref{sec:categories} that we have the following inclusions of categories:   
\[
\Ahol{X}\subseteq \AEmods{X}\subseteq \Amods{X}\subseteq \Dmods{X}. 
\]
\smallskip

\begin{theorem}
\label{thm:invariants-are-good} 
If $M$ is in $\AEmods{X}$, then $\Delta^{\tilde A}_B(M)$ is a finitely 
generated $D_Z$-module, so that we have a restricted functor
\[
\Delta^{\tilde A}_B \colon \AEmods{X}\to\Dmods{Z}.
\]
Further, the following statements hold for $M$ in $\AEmods{X}$. 
\begin{enumerate}
\vspace{-2mm}
\item \label{item:L-hol}
The module $M$ is holonomic if and only if $\Delta^{\tilde A}_B(M)$ is holonomic.
\item \label{item:regular-hol}
  A module $M$ in $\Ahol{X}$ is regular holonomic if and only
  if $\Delta^{\tilde A}_B(M)$ is regular holonomic.
\item \label{item:irred-monodromy} 
  If $M$ has reducible monodromy representation, then so does $\Delta^{\tilde A}_B(M)$. 
  If the columns of $B$ span $\ker_{\ZZ}(A)$ as a lattice, then the converse also holds. 
\end{enumerate}
\end{theorem}

\begin{proof}
This proof occupies the remainder of this section and will be
accomplished through a series of reductions.  

\begin{lemma}
\label{lem:B=C}
If Theorem \ref{thm:invariants-are-good} holds when $K$ from Convention~\ref{conv:BtildeACK} is the
identity matrix, then it holds in general. 
\end{lemma}

\begin{proof}
We wish to show that, for general $K$,
Theorem~\ref{thm:invariants-are-good} holds for the functor
$\Upsilon_K^{K^{-1}}$. 
Up to a coordinate change on the base and range, $\Upsilon_K^{K^{-1}}$
corresponds to the covering map induced by a diagonal matrix. In this
case, $D_{X/\sT}$ is a free $D_Z$-module of 
rank $|\det(K)| = [\ker_{\ZZ}(A):\ZZ B]$, 
and in particular, $\Upsilon_K^{K^{-1}}$ 
preserves finiteness, as desired. 

If $F = (\boldzero_n,\boldone_n)$ and $F' = (\boldzero_m,\boldone_m)$, then 
the $F$-filtered $D_{X/\sT}$-module $M$ corresponds the $F'$-filtered $D_Z$-module $\Upsilon_K^{K^{-1}}(M)$ (which as an underlying
filtered set is exactly $M$). Hence,~\eqref{item:L-hol} of
Theorem~\ref{thm:invariants-are-good} holds for $\Upsilon_K^{K^{-1}}$.

Part~\eqref{item:regular-hol} 
follows from the fact that $\mu_K$ (see Notation~\ref{not:mu}) is an
algebraic coordinate change that induces (locally on $X/\sT$) a
diffeomorphism. Exponential solution growth along the germ of an
embedded curve on $X/\sT$ is then equivalent to such growth along its
image in $Z$. 

For a general $K$ not equal to the identity matrix, only the first half of~\eqref{item:irred-monodromy} applies, and
this follows because $\Delta^{\tilde A}_B$, and thus $\Upsilon_K^{K^{-1}}$, is
exact.  
\end{proof}

We continue with the proof of Theorem~\ref{thm:invariants-are-good},
assuming now that $K$ is the identity matrix. In parallel to the proof
of Theorem~\ref{thm:Ahol sub}, we next prove the theorem in the case of the
simplest diagonal torus action. In this situation, we write $X = T
\times Z$, where $T$ acts on $X$ by scaling the first factor, namely $t
\diamond (s,z) = (ts,z)$, so the matrix $A$ consists of the first $d$
rows of the $n\times n$ identity matrix $\tilde A$.  

\begin{lemma}
\label{lem:TcrossZ-fg}
If $\tilde A=\tilde C$ and $K$ are identity matrices, 
then Theorem~\ref{thm:invariants-are-good} holds.
\end{lemma}

\begin{proof}
Since $\tilde A$ is the identity map in this lemma, we suppress
writing $\Upsilon_C^{\tilde A}$ in the remainder, so that $[M]^T$ represents
$\Delta^{\tilde A}_B(M)$. With a slight abuse of notation, write
$T=\sT$ and $X=T\times Z$.  

Recall that $\AEmods{T\times Z}$, $\Ahol{T\times Z}$, holonomic $A$-graded 
modules on $T\times Z$ with $A$-degree zero morphisms, and regular
holonomic $A$-graded modules on $T\times Z$ with $A$-degree zero
morphisms are all Abelian categories that are closed under
extensions. Thus, in order to show~\eqref{item:L-hol}
and~\eqref{item:regular-hol} of Theorem~\ref{thm:invariants-are-good}
in this case, it is enough by Theorem~\ref{thm:DsT} to consider a
module $N$ for which there exists a $\beta\in\CC^d$ such that for each
nonzero homogeneous element $\gamma\in N$, $\ann_{\CC[E]}(\gamma) 
  = \<t\del_t - \beta - \deg(\gamma)\>$. 
In other words, each $t_i\dt{i}$ acts on $N$ as a multiplication by some scalar 
$\beta_i$.  
Then $D_Z\cdot \gamma =D_Z[E]\cdot \gamma = [D_{T\times Z}]^T\cdot\gamma$
for all $\gamma\in N$.

Since $[N]^T = N_0$ is a finitely generated $[D_{T\times Z}]^T$-module
and $[D_{T\times Z}]^T = D_Z[t\del_t]$, we see that $[N]^T$  
is a finitely generated $D_Z$-module. 
To continue, note that since $N =\bigoplus_{\alpha\in\ZZ A}N_0\cdot t^\alpha$ 
and $D_T = \bigoplus_{\alpha\in\ZZ A} \CC[t\del_t]\cdot t^\alpha$, 
we have that, as $D_{T\times Z}$-modules,
\begin{align}
\label{eq:ann-max-pres}
N = \frac{D_T}{D_T\cdot \<t\del_t - \beta\>} 
	\otimes_{\CC}[N]^T,
\end{align}
with $D_{T\times Z} = D_T\otimes_\CC D_Z$, so that $D_T$ is acting on 
$D_T/(D_T\cdot \<t\del_t - \beta\>)$ and $D_Z$ is acting on $[N]^T$.

For~\eqref{item:L-hol} and~\eqref{item:regular-hol}, 
let $F = (\boldzero_n,\boldone_n) = (\boldzero_d,\boldzero_m,\boldone_d,\boldone_m)$, where $F_T = (\boldzero_d,\boldone_d)$ and $F_Z = (\boldzero_m,\boldone_m)$. 
Then $\left[ \gr^F(D_{T\times Z}) \right]^T 
= \gr^{F_T}( [D_T]^T ) \otimes_\CC \gr^{F_Z}(D_Z)$, 
and~\eqref{eq:ann-max-pres} implies that 
\begin{align}
\gr^F(N) 
  = \gr^F\left( 
  	\frac{ D_T }{ \<t\del_t - \beta\> } 
	\otimes_\CC [N]^T 
	\right) 
   = \gr^{F_T} \left( 
  	\frac{ D_T }{ \<t\del_t - \beta\> } 
	\right)
	\otimes_\CC \gr^{F_Z}( [N]^T ). 
	\label{eq:TtimesZ irred grL}
\end{align}
The module
\begin{align}
\label{eq:T part of N}
H(\beta) \defeq D_T / \<t\del_t - \beta\>. 
\end{align}
is a regular connection on $T$. It follows from~\eqref{eq:TtimesZ
  irred grL} that $N$ is holonomic if and only if $[N]^T$ is holonomic. Thus,~\eqref{item:L-hol} and~\eqref{item:regular-hol}
hold for $N$, from which the general case of the lemma follows.  

We now consider~\eqref{item:irred-monodromy}, still assuming that for
each homogeneous element $\gamma\in N$, $\ann_{\CC[E]}(\gamma) =
\<t\del_t - \beta - \deg(\gamma)\>$. Set 
\[
N(t,z) \defeq \CC(t,z) \otimes_{\CC[t,z]} N \quad \text{and} \quad 
D_{T\times Z}(t,z) \defeq \CC(t,z) \otimes_{\CC[t,z]} D_{T\times Z}.
\]
From~\eqref{eq:ann-max-pres}, we see that 
\begin{align*}
N(t,z) 
& = 
	\left(\left( 
		\CC(t) \otimes_{\CC[t]} H(\beta)
	\right)
	\otimes_{\CC} 
\left( \CC(z) \otimes_{\CC[z]} [N]^T \right)\right)
\otimes_{\CC(t)\otimes_\CC\CC(z)}\CC(t,z). 
\end{align*}
Since $\CC(t)\otimes_{\CC[t]} H(\beta)$ is an irreducible 
$(\CC(t)\otimes_{\CC[t]}D_T)$-module, irreducibility of $N(t,z)$ is
equivalent to that of $\CC(z)\otimes_{\CC[z]}[N]^T$, 
and~\eqref{item:irred-monodromy} holds for $N$. 

For the general case of~\eqref{item:irred-monodromy}, suppose now 
that $M$ is in $\AEmods{X}$ such that for some nonzero $\gamma \in M$,
$\CC[t\del_t]/\ann_{\CC[t\del_t]}(\gamma)$ is Artinian, but
$\ann_{\CC[t\del_t]}(\gamma)$ is not a maximal ideal. Then $M$ has a
filtration by $A$-graded modules  
\begin{align}
\label{eq:filter M}
0 = M^{(0)} \subsetneq M^{(1)} \subsetneq \cdots \subsetneq M^{(r)} = M, 
\end{align}
where $r>1$ and the $\CC[t\del_t]$-annihilators of nonzero homogeneous
elements $\gamma^{(i)}$ in the successive quotients $M^{(i)}/M^{(i-1)}$ have the form 
$\<t\del_t - \beta^{(i)}-\deg(\gamma^{(i)})\>$ for suitable
$\beta^{(i)}$.  Since $M$ is finitely generated, one may choose such a
finite filtration that works simultaneously for all homogeneous
elements of $M$. In particular, $M$ does not have irreducible monodromy representation, since 
$\CC(t,z)\otimes_{\CC[t,z]} M^{(1)}$ provides a nonzero nontrivial submodule of 
$\CC(t,z)\otimes_{\CC[t,z]} M$.  At the same time, applying the exact functor
$\CC(z)\otimes_{\CC[z]}[-]^T$ to~\eqref{eq:filter M} also shows that
$[M]^T$ has reducible monodromy representation, completing the proof
of~\eqref{item:irred-monodromy}, and thus of Lemma~\ref{lem:TcrossZ-fg}.
\end{proof}

Continuing with the proof of Theorem~\ref{thm:invariants-are-good}, 
by Lemma~\ref{lem:B=C} we are left to consider other choices for $\tilde A$ in
Lemma~\ref{lem:TcrossZ-fg} with $K$ still equal to the identity
matrix, so that $Z=X/\sT\cong(\CC^*)^m$ and $D_Z$ acts on $D_X$ via
$C$. Let $X'\defeq (\CC^*)^d\times (\CC^*)^m$ and  
consider the change of coordinates $\mu_{\tilde A}\colon X'\to X$. The
action of $T$ on $X'$ given by identifying $T$ with $(\CC^*)^d$
agrees with the action of $T$ on $X$ through $\mu_{\tilde A}$. At the
same time, $\mu_{\tilde A}$ identifies $X'/T$ with $X/\sT$, so since $K$ is the identity matrix, $Z=X/\sT$.  

If $M$ is in $\AEmods{X}$, then 
$\Delta^{\tilde A}_B(M)
	=\Upsilon_C^{\tilde A}([\mu_{\tilde A}^*(M)]^T)$, 
with the $D_{(\CC^*)^m}$-action coming from the decomposition $X'
=(\CC^*)^d\times (\CC^*)^m$. This implies
parts~\eqref{item:regular-hol} and~\eqref{item:irred-monodromy} of
Theorem~\ref{thm:invariants-are-good}. For part~\eqref{item:L-hol},
note that the monomial map $\mu_{\tilde A}$ identifies the filtration
on $D_X$ induced by $F$ with the order filtration on $D_{X'}$. 
Thus $M$ is holonomic on $X$ if and only if $\mu_{\tilde A}^*(M)$
is holonomic on $X'$. Now by Lemma~\ref{lem:TcrossZ-fg},
$\mu_{\tilde A}^*(M)$ is holonomic on $X'$ if and only 
if $\Upsilon_{{\tilde A}C}([\mu_{\tilde A}^*(M)]^T)$ is holonomic
on $Z =(\CC^*)^m$, completing the proof. 
\end{proof}

We conclude this section with an example to illustrate how a module
$M$ with irreducible monodromy representation could have a reducible image
under $\Delta^{\tilde A}_B$.  

\begin{example}
\label{ex:exp}
Consider the case that $\tilde A$ is the $2\times 2$ identity matrix,
$B=[0\ 2]^t$, $K=[2]$, with
$M=D_X/D_X\cdot\<x_1\dx{1}, \dx{2}-1\>
	=D_X\bullet \exp(x_2)$. 
Here, 
\[
\Delta^{\tilde A}_B(M)
	=\frac{D_Z\oplus D_Z \cdot z^{1/2}}{D_Z \cdot 
	\<(2z\del_z,-1\cdot z^{1/2}),(-z,(2z\del_z-1)\cdot z^{1/2})\>}
\]
is isomorphic to 
$D_Z/D_Z\cdot\<4z\del_z^2+2\del_z-1\>$, 
whose solution space is spanned by 
$\exp(\sqrt z)$ and $\exp(-\sqrt z)$. 
The fundamental group $\ZZ$ of the regular locus of 
$\Delta^{\tilde A}_B(M)$ acts on the solution space by switching the 
two distinguished generators. In particular, 
$f=\exp(\sqrt z)+\exp(-\sqrt z)$ generates a monodromy invariant subspace. As $f$ is holomorphic on $z\not =0,\infty$, it satisfies the
first order equation $(\del_z\cdot\frac{1}{f})\bullet f=0$. 
\endrk
\end{example}

\subsection{Relationship to Descent}
\label{subsec:descent}

The functor $\Delta^{\tilde A}_B$ defined above was
constructed in an attempt to fit the transformation of
Gelfand--Kapranov--Zelevinsky 
hypergeometric modules to hypergeometric modules of Horn type. On the other hand, there exist other
ways of studying equivariant $D$-modules, most notably \emph{descent from $X$}. 
In this section, we explain how $\Delta^{\tilde A}_B$ can be viewed as
a generalization of descent.

Let $G$ be a complex linear algebraic group, and let $X$ be an
algebraic variety with an action of $G$,
\[
\mu\colon G\times X\to X,
\]
such that $\mu(gh,x)=\mu(g,\mu(h,x))$ for all $g,h\in G$ and $x\in
X$. A $\calD_X$-module is weakly $G$-equivariant if there is an
$\calO_G\times D_X$-isomorphism $\phi_M\colon \mu^*(M)\cong
\pi_2^*(M)$ where $\pi_2\colon G\times X\to X$ is the projection.
One calls $M$
an equivariant $D_X$-module if the morphism $\alpha_M$ is
$D_G\otimes D_X$-linear; this is a serious restriction.

\begin{example}
If $G=\CC^*$ and $X=\CC$, with $\mu$ the natural product, then the
modules $D_X/D_X\cdot\<x\del_x-\beta\>$ are equivariant $\calO_X$-modules that
are also $D_X$-modules. However, the only equivariant $D_X$-modules of
this form are $D_X/\<x\del_x-k\>$ where $k\in\ZZ$. Indeed, the Lie algebra of $G$ is spanned
by $x\del_x$. If
\[
M=\bigoplus_{\chi_\alpha\in \hom_{{\rm gps}}(G,\CC^*)}
M_\alpha
\] 
is the decomposition into character spaces and $m\in
M_\alpha$, then $D_X$-linearity implies that
\begin{eqnarray*}
x\del_x\cdot m &=& x\del_x\cdot (1\otimes_\mu m)\\
               &=&\phi_M^{-1}\left(x\del_x\cdot (1\otimes_\CC m)\right)\\
               &=&\phi_M^{-1}\left(x\del_x\bullet(1)\otimes_\CC m + 1\otimes_\CC\chi_\alpha(x\del_x)m\right) \\
               &=&\alpha m,
\end{eqnarray*}
so that the Lie algebra of $G$ (as subset of $D_G\otimes D_X$) acts
via characters on $M$. (The morphism $\chi_\alpha$ induces the one on Lie
algebras used here.)
\endrk
\end{example}

In \cite[Thm VII.12.11]{borel}, it is shown that for group actions
with finitely many orbits, equivariant
$D_X$-modules are regular, underscoring their special
structure.

Now, \cite[Part II]{hotta} discusses $L$-twistedly equivariant
$D_X$-modules. They arise from a rank one connection on $G$, replacing
the trivial local system in the usual equivariant case: one requires
an isomorphism
\[
\phi_L\colon \mu^*(M)\cong L\otimes_\CC M
\]
of $D_G\otimes D_X$-modules on $G\times X$, with appropriate
associativity conditions which boil down to $\mu^*(L)\cong
L\otimes_\CC L$ on $G\times G$. This corresponds to allowing more
flexibility in the action of the Lie algebra of $G$ on $M$. In the
example above, the local connection $(x\del_x-\beta)(f)=0$ on $G$
induces the twistedly equivariant $D_X$-module
$D_X/D_X\cdot \<x\del_x-\beta\>$.

Suppose $P$ is a $G$-bundle over $X$ with structure morphism
$\pi\colon P\to X$. Then the category of $G$-equivariant $D_P$-modules
is equivalent to the category of $D_X$-modules, see Theorem 1.3 in~\cite{raskin}.
The correspondence is induced by taking invariant sections (descent
from $P$ to $X$) and by taking inverse image (from $X$ to $P$).

If in the context of Definition 3.4 we have $X=T\times Z$, then the
functor $\Delta^{\tilde A}_B$ is just descent. In this sense,
$\Delta^{\tilde A}_B$ generalizes descent as follows:
\begin{enumerate}
\vspace{-2mm}
\item it utilizes some additional information (the
matrices $A$, $B$); \label{item:addtl}
\item it provides a descent-like functor for all
twistedly $G$-equivariant $D_X$-modules; \label{item:twisted}
\item it permits the 
adding-on of a finite quotient map $X/T\to Z$. \label{item:finite}
\end{enumerate}

The extra information in~\eqref{item:addtl} is crucial since it distinguishes (the
results of $\Delta^{\tilde A}_B$ on) $D_{\bar X}$-modules that become the same
on $X$. Moreover, the functor $\Delta^{\tilde A}_B$ does not treat the
finite map in~\eqref{item:finite} as part of the 
$T$-action (i.e., the effect is not descent under $X\to Z$; the finite
part of $\Delta^{\tilde A}_B$ does not correspond to taking invariants).

\section{Torus invariants and \texorpdfstring{$D_\barX$}{DbarX}-modules}
\label{sec:Pi}

Let $i\colon X\hookrightarrow \barX$ and $j\colon Z\hookrightarrow
\barZ$ be the natural inclusions. In Convention~\ref{conv:BtildeACK},
we fixed a splitting of $X$, encoded by an $n\times n$ matrix $\tilde
A$, and a Gale dual $B$ of $A$.  
In Theorem~\ref{thm:invariants-are-good}, we examined the functor  
$\Delta^{\tilde A}_B\colon 
	\AEmods{X}\to\Dmods{Z}$, 
which is given by taking torus invariants and adjusting module
structure (see Definition~\ref{def:Delta}).  

\begin{definition}
\label{def:Pi}
The main subject of study in this article is the functor $\Pi^{\tilde A}_B$ on 
$\AEmods{\barX}$ given by 
\[
\Pi^{\tilde A}_B 
\defeq j_+ \circ \Delta^{\tilde A}_B \circ i^*,
\] 
where $j_+$ is the (non-derived) $D$-module direct image and $i^*$ is the $D$-module
inverse image.   
\end{definition}

In this section, we extend our results about $\Delta^{\tilde A}_B$
from \S\ref{sec:invariants} to statements for $\Pi^{\tilde A}_B$. In
\S\ref{sec:E-beta}, we provide further consequences when we further
restrict the source category of $\Pi^{\tilde A}_B$.  

Note that the functor $\Pi^{\tilde A}_B\colon \AEmods{\barX}\to\Dmods{\barZ}$ is exact. 
Indeed, since $i\colon X\to \barX$ is a localization, $i^*$ is exact. Note
also that $i^*$ sends $\AEmods{\barX}$ to $\AEmods{X}$, the 
restricted source category for $\Delta^{\tilde A}_B$.  
Since $\Delta^{\tilde A}_B$ essentially 
takes the $A$-degree $0$ part of a module, it is exact. We thus obtain the
result from the exactness of $j_+$, which holds because $j$ is an open embedding of affine varieties.  

\begin{theorem}
\label{thm:transfer theorem}
If $M$ is in $\AEmods{\barX}$, then $\Pi^{\tilde A}_B(M)$ is a
finitely generated $D_\barZ$-module. In particular, $\Pi^{\tilde
  A}_B\colon \AEmods{\barX}\to\Dmods{\barZ}$, and the following statements hold.  
\begin{enumerate}
\vspace{-2mm}
\item \label{item:L-hol barX}
  	If $M$ is holonomic, 
  	then $\Pi^{\tilde A}_B(M)$ is holonomic.
\item \label{item:regular-hol barX}
	If $M$ in $\Ahol{\barX}$ is regular holonomic, 
	then $\Pi^{\tilde A}_B(M)$ is regular holonomic.
\item \label{item:irred-monodromy barX}
	If $M$ has reducible monodromy representation, then so does 
	$\Pi^{\tilde A}_B(M)$.
  	If the columns of $B$ span $\ker_{\ZZ}(A)$ as a lattice, then
        the converse also holds.
\end{enumerate}
\end{theorem}

\begin{proof}
Since $\sT\subseteq X\subseteq \barX$, 
with input from $\AEmods{\barX}$, $i^*$ returns objects
in $\AEmods{X}$. Therefore by Theorem~\ref{thm:invariants-are-good}
and the fact that the direct image functor $j_+$ will preserve finite
generation, $\Pi^{\tilde A}_B(M)$ is a finitely generated $D_Z$-module
when $M$ is in $\AEmods{\barX}$.  
The properties considered in~\eqref{item:L-hol
  barX},~\eqref{item:regular-hol barX},
and~\eqref{item:irred-monodromy barX} are compatible with the inverse
and direct image functors in the directions of the implications
stated, so the proof reduces to Theorem~\ref{thm:invariants-are-good}.  
\end{proof}

The reverse implications in the first two items of
Theorem~\ref{thm:transfer theorem} do not hold, because the restriction
$i^*$ of a module which is not (regular)  
holonomic might be
(regular) holonomic. In Theorem~\ref{thm:binomial quotients}, we show that
these converses do hold for binomial $D_\barX$-modules.

We conclude this section with a description of the characteristic
varieties and singular loci of $\Pi^{\tilde A}_B(M)$ in terms of those
of $\Delta^{\tilde A}_B(M)$. In \S\ref{subsec:E-beta:charVar Sing}, we
will obtain more refined descriptions of these objects under
additional assumptions. An application of the following
Proposition~\ref{prop:characteristic variety Pi} to the singular locus
of the image under $\Pi^{\tilde A}_B$ of the $A$-hypergeometric system
$H_A(\beta)$ can be found in \S\ref{sec:binomial Sing}.  

\begin{proposition}
\label{prop:characteristic variety Pi}
Let $M$ be a nonzero regular holonomic module in $\AEmods{\barX}$, and
recall that $i\colon X\hookrightarrow\barX$ is the natural inclusion.  
Then the characteristic variety of $\Pi^{\tilde A}_B(M)$ is the union of the 
characteristic variety of $\Delta^{\tilde A}_B(D_X \otimes_{D_\barX} M)$ 
with the coordinate hyperplane conormals in $\barZ$: 
\[
\charVar(\Pi^{\tilde A}_B(M)) 
	\, = \, 
		\charVar(\Delta^{\tilde A}_B(D_X 
		\otimes_{D_\barX} M)) 
		\cup 
		\Var(z_1\zeta_1,\dots,z_m\zeta_m) 
	\ \ \subseteq \ T^*\barZ.
\]
Further, the singular locus of $\Pi^{\tilde A}_B(M)$ is the union of
the singular locus of $\Delta^{\tilde A}_B(D_X \otimes_{D_\barX} M)$
with the coordinate hyperplanes in $\barZ$:  
\[
\Sing{\Pi^{\tilde A}_B(M)}
	\, = \, 
		\Sing{\Delta^{\tilde A}_B(D_X 
		\otimes_{D_\barX} M)} 
		\cup 
		\Var(z_1z_2\cdots z_m)
	\ \ \subseteq \ \barZ.
\]
\end{proposition}
\begin{proof}
Note that $D_X \otimes_{D_\barX} M =  i^* M$. 
Since $\Pi^{\tilde A}_B(M)$ is the direct image of 
$\Delta^{\tilde A}_B(i^*M)$ along the open embedding 
$j\colon Z\hookrightarrow \barZ$,~\cite[Theorem~3.2]{ginsburg} implies 
that the characteristic variety of $\Pi^{\tilde A}_B(M)$ is the union
of the characteristic variety of $\Delta^{\tilde A}_B(i^*M)$ with the
set of coordinate hyperplane conormals in $\barZ$. 
In addition, saturation and projection commute with taking associated
graded objects, so the final statement holds. 
\end{proof}

\section{Torus invariants for a fixed torus character}
\label{sec:E-beta}

In this section, let $Y$ be $X$ or $\barX$. We now restrict our
attention to certain subcategories of $\AEmods{Y}$ called
$\AEmaxmods{Y}$ and $\Tirr(Y,\beta)$, over which we are able to obtain
more detailed information about the functors $\Delta^{\tilde A}_B$ and
$\Pi^{\tilde A}_B$, still following Convention~\ref{conv:BtildeACK}.  
We provide explicit expressions for these functors, show how they
affect solutions, and, under certain assumptions on $B$, describe some
of their geometric properties.  
In Corollary~\ref{cor:Pi of lattice basis is saturated Horn}, we
explain the relationship between lattice basis binomial
$D_\barX$-modules and saturated Horn $D_\barZ$-modules.   

\begin{definition}
\label{def:Tirred}
Let $\AEmaxmods{Y}$ denote the subcategory of $\AEmods{Y}$ consisting
of modules $M$ such that for each nonzero homogeneous element
$\gamma\in M$, $\ann_{\CC[E]}(\gamma)$ is a maximal ideal in
$\CC[E]$. 
\end{definition}
 
To provide a $D$-module theoretic understanding of the definition of
$\AEmods{Y}$, we note the parallel to
Theorem~\ref{thm:DsT}. Recall that each $A$-graded $D_\sT$-module
$M$ has a filtration of $D_\sT$-modules  
\[
0 = M^{(0)} \subset M^{(1)} \subset \cdots 
	\subset M^{(r)} = M 
\]
such that for each $i$ and each nonzero homogeneous element $\gamma$,  
$M^{(i)}/M^{(i-1)}$ has the form $\<E - \beta^{(i)} -\deg(\gamma) \>$ 
for some $\beta^{(i)}$. 
Thus, an $A$-graded $D_\sT$-module $M$ is irreducible if and only if 
there is such a filtration with $r = 1$. 

In the sequel, we consider subcategories of
$\AEmaxmods{Y}$ given by fixing $\beta$. To this end, let
$\Tirr(Y,\beta)$ denote the subcategory of $\AEmaxmods{Y}$ given by
objects $M$ such that for each nonzero homogeneous element 
$\gamma\in M$,   
$\ann_{\CC[E]}(M) 
	= \< E - \beta - \deg(\gamma)\>$. 
Note that when $\gamma\in D_Y$ is homogeneous, 
$\gamma(E-\beta) 
= (E-\beta-\deg(\gamma))\gamma$. 
Thus, when $M = D_Y/J$ is a cyclic element in $\AEmaxmods{Y}$, it lies in
$\Tirr(Y,\beta)$ precisely when $\<E-\beta\>$ is contained in $J$.  

\begin{example}
If $I\subseteq\CC[\del_x]$ is an $A$-graded ideal, then
$D_\barX/(I+\<E-\beta\>)$ is in $\Tirr(\barX,\beta)$.   
\endrk
\end{example}

\subsection{Explicit expressions for the functors $\Delta^{\tilde A}_B$ and
$\Pi^{\tilde A}_B$}
\label{subsec:E-beta:explicit}

We provide explicit computations of $\Delta^{\tilde A}_B$ and
$\Pi^{\tilde A}_B$ on $\Tirr(X,\beta)$. 
We begin by first relating $[D_X]^T/\<E-\beta\>$ and $D_Z$ via
$D_{X/\sT}$.

\begin{convention}
\label{conv:kappa beta}
For each $\beta\in\CC^d$, fix a vector $\kappa(\beta) = \kappa \in \CC^n$ 
such that $A\kappa = \beta$. 
In addition, recall that the $m\times m$ matrix $K$ from
Convention~\ref{conv:BtildeACK} has full rank, and therefore its Smith
normal form is an integer diagonal matrix with nonzero diagonal
entries called the \emph{elementary divisors} of $K$. Let $\varkappa \defeq
(\varkappa_1,\dots,\varkappa_m)$ denote the elementary divisors of $K$. 
\endrk
\end{convention}

\begin{proposition}
\label{prop:dkap ker}
Let $B$ be a Gale dual of $A$ whose columns span
$\ker_{\ZZ}(A)$ as a lattice,  
so that $[D_X]^T$ is $\CC$-spanned by monomials $x^{Bv} \theta^u$, 
where $v\in \ZZ^m$ and $u \in \NN^n$. Denote the rows of $B$ by
$B_1,\dots,B_n$. For $v\in\ZZ^m$ and $u \in \NN^n$, define   
\begin{align}\label{eq:dkap}
\dkap( x^{Bv} \theta^u ) 
	\defeq z^v \prod_{i=1}^n (B_i\cdot\eta + \kappa_i)^{u_i}.
\end{align}
This extends linearly to a surjective homomorphism of $\CC$-algebras
\[
\dkap: [ D_X ]^T \longrightarrow D_Z,
\]
whose kernel is the (two-sided) $[D_X]^T$-ideal generated by the
sequence $E-\beta$. 
\end{proposition}

We will address the implications of different choices for $\kappa$ later, in Remark~\ref{rem:differentKappas}. 

\begin{corollary}
\label{cor:dkap in general}
Let $B$ be any Gale dual of $A$, and let $\kappa$ and
$\varkappa$ be as in Convention~\ref{conv:kappa beta}, and set
$\varepsilon_C \defeq \sum_{i=1}^m c_i$, where $c_i$ is the $i$th column of $C$. Then there is an isomorphism,
denoted $\overline{\delta}_{B,\kappa}$, given by the composition of  
\[
\delta_{C,\kappa+\varepsilon_C}\colon 
\frac{[ D_X ]^T}{\<E-\beta\>} \stackrel{\sim}{\longrightarrow} 
D_Z 
\]
and the isomorphism induced by $\mu_K$: 
\begin{equation}
\label{eqn:iso of rings}
D_Z 
\ \cong \ 
\bigoplus_{0\leq k<\varkappa}D_Z\cdot z^{k/\varkappa},
\end{equation}
where the comparison $k<\varkappa$ is  component-wise and $k/\varkappa$ 
is the vector
$(k_1/\varkappa_1,\ldots,k_m/\varkappa_m)$. 
\end{corollary}
\begin{proof}[Proof of Corollary~\ref{cor:dkap in general}] 
The map $\delta_{C,\kappa+\varepsilon_C}$ is an isomorphism by
Proposition~\ref{prop:dkap ker}, so it is enough to
understand~\eqref{eqn:iso of rings}. First note that by identifying $y_i$ with $z_i^{1/\varkappa_i}$, 
\[
\bigoplus_{0\leq k<\varkappa}D_Z\cdot z^{k/\varkappa}
\ \cong \ 
\frac{D_Z[y_1,\dots,y_m]}{\<y_i^{\varkappa_i}-z_i \mid i= 1,\dots,m\>}. 
\]
By Convention~\ref{conv:kappa beta}, there are matrices $P, Q\in\Gl(m,\ZZ)$ such that $P K Q$ is the
diagonal matrix whose diagonal entries are the components of $\varkappa$.
The maps $\mu_P$ and
$\mu_Q$ both induce isomorphisms of $D_Z$; for the matrix $P$, whose 
columns are denoted by $p_1,\dots,p_m$, this isomorphism is given by $z_i
\mapsto z^{p_i}$ and $\eta_i \mapsto \sum_{j=1}^m
\overline{p}_{ij} \eta_j$, where $P^{-1} = [\overline{p}_{ij}]$.

Since $P$ and $Q$ induce isomorphisms, we may assume that $K$ is
diagonal with diagonal entries $\varkappa_1,\dots,\varkappa_m$. In
this case, the ring homomorphism $D_Z \to
D_Z[y_1,\dots,y_m]/\<y_i^{\varkappa_i}-z_i \mid i=1,\dots m\>$ 
by $z_i \mapsto y_i$ and $\eta_i \mapsto \varkappa_i \eta_i$ is clearly an isomorphism.
\end{proof}

\begin{proof}[Proof of Proposition~\ref{prop:dkap ker}]
If $Bv = Bv'$, then $v=v'$ because $B$ has full rank. Moreover, since
the columns of $B$ span $\ker_{\ZZ}(A)$ as a lattice, the elements
$x^{Bv}\theta^u$ form a basis of $[D_X]^T$ as a $\CC$-vector
space. Thus $\dkap$ is well-defined.  

For $\dkap$ to be a ring homomorphism, it is enough to show that 
$\dkap( \theta_i^k x^{Bv}) = \dkap(\theta_i^k)\dkap(x^{Bv})$. 
This follows from two key identities that hold for any $1\leq i\leq n$
and $k\in\ZZ$:  
\begin{align*}
\theta_i^k x^{Bv} 
= x^{Bv}(\theta_i + B_i\cdot v)^k 
\qquad\text{and}\qquad
[ B_i\cdot\eta + \kappa_i ]^k z^v 
= z^v [ B_i\cdot v + B_i\cdot\eta+\kappa_i ]^k. 
\end{align*}
For surjectivity, observe that $z_i^{\pm 1} = \dkap(x^{\pm Be_i})$. 
We thus need to show that $\eta_1,\dots,\eta_m$ belong to the image of $\dkap$.
For notational convenience, assume that the first $m$ rows of $B$ are 
linearly independent. Call $N$ the corresponding submatrix of $B$. 
Let $N_i^{-1}$ denote the $i$th row of $N^{-1}$. Then 
$\dkap( N_i^{-1} \cdot [\theta_1-\kappa_1,\dots,\theta_m-\kappa_m] ) 
= \eta_i$. 

Now consider $F \in \ker(\dkap)$. Since $F \in [D_X]^T$, we can write 
$F = \sum_i x^{B v_i}p_i(\theta)$, where the $v_i$ are distinct, the
$p_i$ are polynomials in $n$ variables, and the sum is finite. Then
$\dkap(F) = \sum_i z^{v_i}p_i(B\eta + \kappa) = 0$. 
Since the zero operator annihilates every monomial $z^{\mu}$ with $\mu\in \ZZ^m$,
$0 = \dkap(F) \bullet z^{\mu} 
= \sum_i z^{v_i} p_i(B\mu + \kappa) z^{\mu}, 
$ and, since the $v_i$ are distinct, 
$p_i(B\mu + \kappa) = 0$ for all $\mu \in \ZZ^m$ and for all $i$. 
Hence each $p_i$ vanishes on the Zariski closure of 
$\kappa + \ker_{\ZZ}(A)$, so by the Nullstellensatz, every
$p_i( \theta)$ is an element of $\CC[\theta] \cdot \<E-\beta \>$.   
It follows that $F$ belongs to the $[D_X]^T$-ideal generated by $E-\beta$, as desired.
\end{proof}

\begin{example}
\label{ex:dkap in general}
Consider the matrices 
\[
\tilde A = 
\begin{bmatrix} 1&1&1&1\\ 
0&1&2&3\\
1&0&0&0\\ 
0&1&0&0 \end{bmatrix}, 
\quad 
C = 
\begin{bmatrix} \phantom{-}1&\phantom{-}0\\ 
\phantom{-}0&\phantom{-}1\\
-3&-2\\ 
\phantom{-}2&\phantom{-}1 \end{bmatrix}, 
\quad 
B = 
\begin{bmatrix} -1 & \phantom{-}2 \\
\phantom{-}0 & -3 \\
\phantom{-}3 & \phantom{-}0 \\
-2 & \phantom{-}1 \end{bmatrix},
\quad \text{ and }\quad 
K = 
\begin{bmatrix} -1&\phantom{-}2\\ 
\phantom{-}0&-3 \end{bmatrix}.
\]
From Corollary~\ref{cor:dkap in general}, using $y_2^3 = 1/(z_1^2z_2)$, 
we obtain the isomorphism 
\[
\hspace{2.7cm}
\frac{[D_X]^T}{\<E-\beta\>} \ \cong \ 
D_{X/\sT} \ \cong \ 
D_Z \, \oplus \, D_Z \cdot \left(\frac{1}{z_1^2z_2}\right)^{1/3} \,
\oplus \, D_Z \cdot \left(\frac{1}{z_1^2z_2}\right)^{2/3}. 
\hspace{2.4cm}
\hexagon
\]
\end{example}

Let $\overline{\delta}_{B,\kappa}\colon {[D_X]^T}/{\<E-\beta\>}\to D_Z$ denote the isomorphism induced by $\delta_{B,\kappa}$ from Proposition~\ref{prop:dkap ker}. 
For a module $M$ in $\Tirr(X,\beta)$ or $\Tirr(\barX,\beta)$, $\overline{\delta}_{B,\kappa}$ can be used to explicitly compute 
$\Delta^{\tilde A}_B(M)$ or $\Pi^{\tilde A}_B(M)$, respectively. 

\begin{corollary}
\label{coro:invariants via dkap}
Suppose that $\phi\colon (D_X)^p\to (D_X)^q$ is a presentation matrix for a 
module $M$ in $\Tirr(X,\beta)$, and let $\kappa$ and $\varkappa$ be as
in Convention~\ref{conv:kappa beta}. Then the $D_Z$-module $\Delta^{\tilde A}_B(M)$ 
is presented up to isomorphism by  
\[
\overline{\delta}_{B,\kappa} \left( \frac{[D_X]^T}{\<E-\beta\>} 
	\otimes_{[D_X]^T} \Delta^{\tilde A}_B(\phi) \right)\colon 
	(D_Z)^p\longrightarrow \left(\bigoplus_{0\leq k  <\varkappa} D_Z\cdot z^{k/\varkappa}\right)^q. 
\]
In particular, if $J +\<E-\beta\>$ is a torus equivariant left 
$D_X$-ideal, then there is an isomorphism
\[
\hspace{4.5cm}
\Delta^{\tilde A}_B\left( \frac{D_X}{J+\<E-\beta\>} \right) \cong
\frac{\bigoplus_{0\leq k<\varkappa} D_Z\cdot
  z^{k/\varkappa}}{\dkap(J)}. 
\hspace{3.75cm}
\qed
\]
\end{corollary}

\begin{proof}
If $K$ is not invertible over $\ZZ$ and $M$ is irreducible, change
coordinates so that $K$ is diagonal with diagonal entries
$\varkappa$. 
Then write (in the new coordinates) 
$M =(\bigoplus_\epsilon D_{X/\sT} 
	\cdot \epsilon)/I$. 
Applying Corollary~\ref{cor:dkap in general} and sorting by (rational)
powers of $z$,  
$M=(\bigoplus_{0\le k< \varkappa}
  \bigoplus_\epsilon D_Z \cdot \epsilon 
  \cdot z^{k/\varkappa})/I$, 
where $I$ is being expanded in terms of the $D_Z \cdot z^{k/\varkappa}$.
The result now follows from Corollary~\ref{cor:dkap in general}. 
\end{proof}

\begin{corollary}
\label{coro:quotient via dkap}
Suppose that $\phi\colon (D_\barX)^p\to (D_\barX)^q$ is a presentation
matrix for a module $M$ in $\Tirr(\barX,\beta)$, and let $\kappa$ and
$\varkappa$ be as in Convention~\ref{conv:kappa beta}. Then the
$D_\barZ$-module $\Pi^{\tilde A}_B(M)$ is isomorphic to the quotient
of $(D_\barZ)^q$ by  
\[
\left(\bigoplus_{0\leq k<\varkappa} D_\barZ\cdot z^{k/\varkappa}\right)^q 
\ \cap \ 
\image\left( 
	\overline{\delta}_{B,\kappa} \left( \frac{[D_X]^T}{\<E-\beta\>} 
	\otimes_{[D_X]^T} D_X\cdot \Delta^{\tilde A}_B(\phi) \right) 
	\right). 
\]
In particular, if $I+\<E-\beta\>$ is a torus equivariant left 
$D_\barX$-ideal, then there is an isomorphism
\[
\hspace{3cm}
\Pi^{\tilde A}_B\left( \frac{D_\barX}{J+\<E-\beta\>} \right) \cong
\frac{\bigoplus_{0\leq k<\varkappa} D_\barZ\cdot
  z^{k/\varkappa}}{\dkap(D_X\cdot J)\cap \left(\bigoplus_{0\leq
      k<\varkappa} D_\barZ\cdot z^{k/\varkappa}\right)}. 
\hspace{2.2cm}
\qed
\]
\end{corollary}

\begin{remark}
\label{rem:differentKappas}
Note that if we change our choice of $\kappa$ in
Convention~\ref{conv:kappa beta}, we obtain different isomorphisms
in Proposition~\ref{prop:dkap ker} and Corollary~\ref{cor:dkap in
  general}.  Thus, different choices of  $\kappa$ will produce
(isomorphic) presentations of $\Delta^{\tilde A}_B(M)$  
(respectively, $\Pi^{\tilde A}_B(M)$) if $M$ is an element of
$\Tirr(X,\beta)$ (respectively, $\Tirr(\barX,\beta)$).  
Likewise, different choices of Gale duals $B$ will also produce different
isomorphisms, and therefore different presentations. 
\endrk
\end{remark}

\begin{example}[Example~\ref{ex:dkap in general}, first variant]
\label{ex:GKZ under Pi}
We use Corollary~\ref{coro:quotient via dkap} to apply $\Pi^{\tilde
  A}_B$ to the $A$-hyper\-geo\-met\-ric $D_\barX$-module associated to $A$
and $\beta$. 
Let $C_1,\dots,C_4$ denote the rows of $C$. 
Under $\overline{\delta}_{C,\kappa+\varepsilon_C}$, the generators of 
$I_A = 
	\<\del_{x_3}^2-\dx{2}\dx{4}, 
	\dx{2}\dx{3}-\dx{1}\dx{4}, 
	\del_{x_2}^2 - \dx{1}\dx{3}\>$ 
map respectively to 
\begin{align}
\nonumber
&(C_3\cdot \lambda + \kappa_3-2)
	(C_3\cdot \lambda + \kappa_3-3) 
	- y_2
	(C_2\cdot \lambda+\kappa_2+1)
	(C_4\cdot \lambda + \kappa_4+1), \\
\nonumber
&(C_2\cdot \lambda + \kappa_2+1)
	(C_3\cdot \lambda + \kappa_3-2)
	-\frac{y_2}{y_1}
	(C_1\cdot \lambda + \kappa_1)
	(C_4\cdot \lambda + \kappa_4+1), 
	\quad\text{and}
	\\
\nonumber
&(C_2\cdot \lambda+\kappa_2+1)
	(C_2\cdot \lambda+\kappa_2)
	-\frac{y_2^2}{y_1}
	(C_1\cdot \lambda + \kappa_1)
	(C_3\cdot \lambda + \kappa_3-2).
\intertext{Thus, under $\overline{\delta}_{B,\kappa}$, the binomial generators of $I_A$ map to }
\label{eq:image1}
&(B_3\cdot \eta + \kappa_3)
	(B_3\cdot \eta + \kappa_3-1) 
	- \left(\frac{1}{z_1^2z_2}\right)^{\frac{1}{3}}
	(B_2\cdot \eta+\kappa_2)
	(B_4\cdot \eta + \kappa_4), \\
\label{eq:image2}
&(B_2\cdot \eta + \kappa_2)
	(B_3\cdot \eta + \kappa_3)
	-\left(\frac{z_1}{z_2}\right)^{\frac{1}{3}}
	(B_1\cdot \eta + \kappa_1)
	(B_4\cdot \eta + \kappa_4), 
	\quad\text{and}
	\\
\label{eq:image3}
&(B_2\cdot \eta+\kappa_2)
	(B_2\cdot \eta+\kappa_2-1)
	-\left(\frac{1}{z_1z_2^2}\right)^{\frac{1}{3}}
	(B_1\cdot \eta + \kappa_1)
	(B_3\cdot \eta + \kappa_3).
\end{align}
Going modulo these equations in 
$D_\barZ \, \oplus \, 
D_\barZ \cdot \left(\frac{1}{z_1^2z_2}\right)^{1/3} \, \oplus \, 
D_\barZ \cdot \left(\frac{1}{z_1^2z_2}\right)^{2/3}$, we obtain
$\Pi^{\tilde A}_B(H_A(\beta))$ by Corollary~\ref{coro:quotient via
  dkap}.  
\endrk
\end{example} 

With Corollary~\ref{coro:quotient via dkap} in hand, we now have the
tools to relate lattice basis binomial $D_\barX$-modules and saturated
Horn $D_\barZ$-modules via $\Pi^{\tilde A}_B$. 
This result will be further exploited in Part~II. 

\begin{corollary}
\label{cor:Pi of lattice basis is saturated Horn}
If $\kappa$ is as in Convention~\ref{conv:kappa beta}, then 
\[
\Pi^{\tilde A}_B\left(\frac{D_\barX}{I(B)+\< E-\beta\>}\right) 
\ \cong \ 
\bigoplus_{0\leq k<\varkappa} \frac{D_\barZ}{\sHorn(B,\kappa+Ck)}
\cdot z^{k/\varkappa}.
\]
\end{corollary}
\begin{proof}
If the columns of $B$ span 
$\ker_{\ZZ}(A)$ as a lattice, then the
statement follows immediately from the definitions of the systems, via
Corollary~\ref{coro:quotient via dkap}. 
For an arbitrary choice of Gale dual $B$, we again apply
Corollary~\ref{coro:quotient via dkap}. Since passage of
$z^{k/\varkappa}$ through the equations of $\sHorn(B,\kappa)$ results
in a shift by $Ck$, the summands separate to yield the desired
isomorphism.  
\end{proof}

\begin{example}[Example~\ref{ex:GKZ under Pi}, continued]
Since~\eqref{eq:image1} does not lie in a single  $D_Z$-summand of
$[D_X]^T/\<E-\beta\>$, we see immediately that $\Pi^{\tilde
  A}_B(D_\barX/(I_A+\<E-\beta\>)$ does not possess the nice
decomposition that arises in the lattice basis situation.  
After multiplying~\eqref{eq:image2} and~\eqref{eq:image3} by
$z_1^{-1}$, it is also clear that neither of these elements lie in a
single $D_Z$-summand of $[D_X]^T/\<E-\beta\>$. 
\endrk
\end{example}

\begin{example}[Example~\ref{ex:dkap in general}, second variant] 
\label{ex:Pi of lattice basis is saturated Horn}
We now use Corollary~\ref{coro:quotient via dkap} to compute the module 
$\Pi_B^{\tilde A}(D_\barX/(I(B)+\<E-\beta\>))$. 
In this example,  
$I(B) = 
\< \del_{x_3}^3-\del_{x_1}\del_{x_4}^2, \ 
\del_{x_1}^2\del_{x_4}-\del_{x_2}^3\>$. 
Since $\varepsilon_C = [1, 1, -5, 3]^t$, 
under $\overline{\delta}_{B,\kappa}$, the  binomials generating $I(B)$
map respectively to  
\begin{multline}
\label{eq:image Y lattice basis}
(C_3\cdot\lambda+\kappa_3-5)
(C_3\cdot\lambda+\kappa_3-6)
(C_3\cdot\lambda+\kappa_3-7)
\\
-y_1^{-1}
(C_1\cdot\lambda+\kappa_1+1)
(C_4\cdot\lambda+\kappa_4+3)
(C_4\cdot\lambda+\kappa_4+2)
\quad\text{and}
\shoveright \ 
\\
\shoveleft
(C_1\cdot\lambda+\kappa_1+1)
(C_1\cdot\lambda+\kappa_1)
(C_4\cdot\lambda+\kappa_4+3)
\\
-\frac{y_1^2}{y_2^3}
(C_2\cdot\lambda+\kappa_2+1)
(C_2\cdot\lambda+\kappa_2)
(C_2\cdot\lambda+\kappa_2-1).
\end{multline}
Then under $\overline{\delta}_{B,\kappa}$, the expressions
of~\eqref{eq:image Y lattice basis} map to the generators of
$\sHorn(B,\kappa)$, while  $y_2$ times the expressions
of~\eqref{eq:image Y lattice basis} map respectively to 
\begin{multline}
\label{eq:image y2 dkap lattice basis}
[
(B_3\cdot\eta+\kappa_3-2)
(B_3\cdot\eta+\kappa_3-3)
(B_3\cdot\eta+\kappa_3-4)\\
-z_1
(B_1\cdot\eta+\kappa_1)
(B_4\cdot\eta+\kappa_4+1)
(B_4\cdot\eta+\kappa_4)
]
\cdot \left(\frac{1}{z_1^2z_2}\right)^{\frac{1}{3}}
\quad\text{and}
\shoveright \ 
\\
\shoveleft
[
(B_1\cdot\eta+\kappa_1)
(B_1\cdot\eta+\kappa_1-1)
(B_4\cdot\eta+\kappa_4+1)\\
-z_2
(B_2\cdot\eta+\kappa_2+1)
(B_2\cdot\eta+\kappa_2)
(B_2\cdot\eta+\kappa_2-1)
]
\cdot \left(\frac{1}{z_1^2z_2}\right)^{\frac{1}{3}}. 
\end{multline}
Similarly, we can compute $y_2^2$ times the expressions
of~\eqref{eq:image Y lattice basis}. Together, these show that the
image of the lattice basis binomial $D_\barX$-module
$D_\barX/(I(B)+\<E-\beta\>)$ under $\Pi^{\tilde A}_B$ is isomorphic to  
\begin{align*}
\frac{D_\barZ}{\sHorn(B,\kappa)}
\ \oplus \ 
\frac{D_\barZ}{\sHorn(B,\kappa+c_2)} 
	\cdot \left(\frac{1}{z_1^2z_2}\right)^{\frac{1}{3}}  
\ \oplus \ 
\frac{D_\barZ}{\sHorn(B,\kappa+2c_2)} 
	\cdot \left(\frac{1}{z_1^2z_2}\right)^{\frac{2}{3}}.
\end{align*}
This agrees with 
Corollary~\ref{cor:Pi of lattice basis is saturated Horn}. 
\endrk
\end{example}

\subsection{Solution spaces}
\label{subsec:E-beta:solutions}

We now consider the effect of $\Pi^{\tilde A}_B$ on solutions of objects in $\AEmaxmods{\barX,\beta}$. 
Suppose that a module $M$ in $\Tirr(\barX,\beta)$ is such that 
at a sufficiently generic (nonsingular) point $p\in X$, 
$\Sol_p(M)$ has a basis of basic Nilsson solutions in the direction of
a generic weight vector $w\in\RR^n$.   
(This occurs when $M$ is regular holonomic, but regularity is not a
necessary condition, see Theorems~\ref{thm:solutions-of-regular-holonomic} and~\ref{thm:nilsson binomial enough}.)   
Then $\Sol_p(M)$ is spanned by vectors of the form $\phi =
(\phi_1,\dots,\phi_r)$, where the $\phi_i$ are Nilsson solutions that
converge at $p$.  
Further, 
by Remark~\ref{rem:Agrd nilsson}, 
any solution $\phi$  of $M$ can be written $\phi = x^\kappa f(x^\sL)$,
where $f$ is a function in $m$ variables and $\sL$ is a collection of $m$ vectors that $\ZZ$-span
$\ker_{\ZZ}(A)$.   

Using $b_1,\dots,b_m$ to denote the columns of a Gale dual $B$ of $A$,
the map $\mu_B\colon X\to Z$ given by $x\mapsto x^B \defeq
(x^{b_1},\dots,x^{b_m})$ is onto, since $\CC$ is algebraically closed. 
If $p\in X$, this surjection induces a map 
$\mu_B^\sharp\colon \calO^{\text{an}}_{Z,p^B} \to \calO^{\text{an}}_{X,p}$ 
via $g(z)\mapsto g(x^B)$. 
If the columns of $B$ span $\ker_{\ZZ}(A)$ as a lattice, then
$\mu_B^\sharp$ is an isomorphism.  
Otherwise, it will be a $[\ker_{\ZZ}(A):\ZZ B]$-to-$1$ covering
map. Let $s_1,\dots,s_{[\ker_{\ZZ}(A):\ZZ B]}$ denote the distinct
sections of this map.  

\begin{theorem}
\label{thm:solutions E}
Let $M$ be a module in $\Tirr(\barX,\beta)$ 
and let $p\in X$ be a generic 
nonsingular point of $M$ so that $\mu_B(p) = p^B\in Z$ is 
a nonsingular point of $\Pi^{\tilde A}_B(M)$. 
Choose $\kappa$ so that $A\kappa = \beta$ and $\kappa C = 0$. 
If $\Sol_p(M)$ 
has a basis of basic Nilsson solutions in the direction of a generic
weight vector, then $\Sol_{p^B}(\Pi^{\tilde A}_B(M))$ is isomorphic to
the sum over $i\in\{1,\dots,[\ker_{\ZZ}(A):\ZZ B]\}$ of the images of
the maps   
\begin{align*}
\Sol_p(M)
&\longrightarrow
\Sol_{p^B}(\Pi^{\tilde A}_B(M))     
\quad \text{given by} \quad 
x^\kappa f(x^\sL) \, \ \mapsto\ f(s_i(z)). 
\end{align*}
\end{theorem}

\begin{corollary}
\label{cor:rank Pi}
If $M$ is a regular holonomic module in $\Tirr(\barX,\beta)$, 
then the rank of $\Pi^{\tilde A}_B(M)$ is equal to $[\ker_{\ZZ}(A):\ZZ B]\cdot\rank(M)$. 
Moreover, if the columns of $B$ span $\ker_{\ZZ}(A)$ as a lattice 
(so that $[\ker_{\ZZ}(A):\ZZ B] = |\det(K)| = 1$) 
and $p\in X$ is a generic nonsingular point of $M$, then
there is an isomorphism $\Sol_p(M) \cong \Sol_{p^B}(\Pi^{\tilde A}_B(M))$. 
\qed
\end{corollary}

\begin{proof}[Proof of Theorem~\ref{thm:solutions E}]
Since $p\in X$ is nonsingular for $M$, restriction provides an
isomorphism between the solutions of $M$ at $p$ and those of 
$i^* M = D_X\otimes_{D_\barX} M$ at $p$.
By the genericity of $p$, $p^B \in Z$ will be a nonsingular point of  
$\Delta^{\tilde A}_B(i^*M)$, and moreover, the solutions of
$\Delta^{\tilde A}_B(i^* M)$ at $p^B$ can be extended to solutions of
$\Pi^{\tilde A}_B(M)$ at $p^B$. This means that we may assume that $M$ is a
$D_X$-module and work with the functor $\Delta^{\tilde A}_B$ instead of 
$\Pi^{\tilde A}_B$. 
Further, as solutions of $M$ are represented by vectors of
functions, for simplicity, it is enough to consider case that $M$ is
cyclic.  

Consider first the case that the columns of $B$ do not $\ZZ$-span
$\ker_{\ZZ}(A)$. Then $\mu_K\colon X/\sT \to Z$ induces a
$|\det(K)|$-fold cover of $Z$, which induces an isomorphism between
the solutions of $\Delta^{\tilde A}_B(M)$ at $p^B$ and  
the sum over all sections of $\mu_K$ of the image of the solution
space of $\Delta^{\tilde A}_C(M) = \Upsilon_C^{\tilde A}([M]^T)$. Thus, we
have reduced the proof of the theorem to the case that the columns of
$B$ span $\ker_{\ZZ}(A)$ as a lattice, so that $\mu_B$ is an
isomorphism. We assume this case for the remainder of the proof; in
particular,  without loss of generality, we may assume that $B=C=\sL$.  

Consider now the case that $X=T'\times Z$, where $T' = (\CC^*)^d$, 
and $T'\times Z$ has a torus action given by 
$t \diamond (s,z) = (t_1^{a_{1,1}}s_1,\dots,t_d^{a_{d,d}}s_d,z_1,\dots,z_n)$. 
Let $E'$ denote the Euler operators on $T'\times Z$, 
viewed as elements in $D_{T'}$. 
If $N$ is a module in $\Tirr(T'\times Z,\beta)$, then by the same
argument used to obtain~\eqref{eq:ann-max-pres}, $N$ can be expressed
as a module over $D_{T'\times Z} = D_{T'} \otimes_\CC D_Z$ as 
\begin{align}
\label{eq:ann-max-pres-general}
N = 
\frac{ D_{T'} }{ \<E'-\beta\> } \otimes_\CC 
\Upsilon_C^{\tilde A}([N]^T)
= 
\frac{ D_{T'} }{ \<E'-\beta\> } \otimes_\CC 
\Delta^{\tilde A}_C(N).
\end{align}
Suppose that $(t_0,z_0)\in T'\times Z$ is a generic nonsingular point of $N$. 
From~\eqref{eq:ann-max-pres-general}, we see that 
\begin{align}
\Sol_p(N) &= 
\hom_{D_{T'\times Z}}( N, \calO^{\text{an}}_{T'\times Z, p} ) \nonumber\\
	&\cong 
	\hom_{D_{T'}}\left(\frac{D_{T'}}{\<E'-\beta\>},\calO^{\text{an}}_{T',t_0}\right)
	\otimes_\CC
	\hom_{D_Z}\left( 
	\Delta^{\tilde A}_C(N), 
	\calO^{\text{an}}_{Z,z_0} \right).
\label{eq:T'timesZ sol isom}
\end{align}
For the isomorphism~\eqref{eq:T'timesZ sol isom}, use 
the fact that
$\hom_{D_{T'}}\left({D_{T'}}/{\<E'-\beta\>},\calO^{\text{an}}_{T',t_0}\right)$
is spanned by the homomorphism 
$1\mapsto\prod_{i=1}^d t_i^{\beta_i/a_{i,i}}$. 
In particular, it is a one-dimensional $\CC$-vector space.

We now return to the general case that $M$ is in $\Tirr(X,\beta)$. 
Let $\phi\colon X\to T'\times Z$ be a $T$-equivariant change of
coordinates, whose existence follows from the existence of a Smith 
normal form for $\tilde A$. 
Let $C$ and $C'$ denote the appropriate $C$-matrices for $X$ and
$T'\times Z$, respectively. 
With $(t_0,z_0) \defeq \phi(p)$, since $\kappa C = 0$, we have the
following isomorphisms: 
\begin{align*} 
\Sol_p(M) 
  &\cong \Sol_{\phi(p)}(\phi_+ M)&\\
  &\cong \Sol_{t_0}((\phi_+M) |_{T'} ) 
    \otimes_\CC \Sol_{z_0}( 
    \Delta^{\tilde A}_{C'}(\phi_+M) ) 
    &\text{\hspace{-7mm}(using an argument similar 
    to~\eqref{eq:ann-max-pres-general})}\\ 
  &\cong \Sol_{t_0}( D_{T'} / \<E' - \beta\> ) 
    \otimes_\CC \Sol_{z_0} ( 
    \Delta^{\tilde A}_{C'}(\phi_+M) ) 
	&\text{(by applying~\eqref{eq:T'timesZ sol 
	isom} to $\phi_+(M)$)} \\
  &\cong \Sol_{(t_0^{a_1},\dots,t_0^{a_n})}
    ( D_\sT/\<E-\beta\> ) \otimes_\CC 
    \Sol_{p^B}(
    \Delta^{\tilde A}_{C}(M) ) 
    &\\ 
  &\cong \Sol_{p^B}(
    \Delta^{\tilde A}_{C}(M) ). 
  &
\end{align*}
The final equality follows since 
$\Sol_{(t_0^{a_1},\dots,t_0^{a_n})}( D_\sT/\<E-\beta\> )$
is spanned by the function $x^{\kappa}$. In fact, it is spanned by any 
function $x^v$ such that $Av=\beta$, but observe that if $Au=0$, then
$x^u\equiv 1$ on $\sT$.

Finally, note that two different versions of taking invariants, $[-]^T$ and $[-]^{T'}$, are used in the previous display when respectively applying $\Delta^{\tilde A}_C$ and $\Delta^{\tilde A}_{C'}$; however,
they are compatible via the isomorphism $\phi$ because 
$\mu_B = \pi_2\circ \phi\colon X\to T'\times Z\to Z$.  
\end{proof}

\begin{example}
\label{ex:weird solutions}
The assumption that the module $M$ belongs to $\Tirr(\barX,\beta)$ is
crucial for Theorem~\ref{thm:solutions E}. The key fact here is that
$\<E-\beta\>$ is a maximal ideal in $\CC[E]$.

Consider instead the case that $d=m=1$ and $X = T\times Z$, 
where $T$ acts by scaling on the first factor. 
The torus equivariant module 
$M \defeq 
D_{T\times Z}/D_{T\times Z}\cdot \< (t\del_t)^3, t\del_t - \del_z \>$ 
is holonomic, and its solution space at a nonsingular point $(t_0,z_0)$ is 
\[
\Sol_{(t_0,z_0)}(M) = \Span_\CC \{ 1,\ z + \log(t),\ \log^2(t) + 2z\log(t) \}. 
\]
In this case, since $\ann_{\CC[E]}(M) = \<(t\del_t)^3\>$, the 
operator $t\del_t$ does not act as a constant on $M$, and therefore, the 
operator $t\del_t-\del_z$ cannot be written modulo $\ann_{\CC[E]}(M)$ as 
an element of $D_Z$. This is the reason that the solutions of $M$ are
not as well-behaved as in Theorem~\ref{thm:solutions E}.
\endrk
\end{example}

\begin{remark}
\label{rem:solutions Horn}
Note that in the proof of 
Theorem~\ref{thm:solutions E}, $\mu_K$ is used to describe the
behavior of solutions when the columns of $B$ do not $\ZZ$-span
$\ker_{\ZZ}(A)$. Combining this argument with the direct sum
decomposition in Corollary~\ref{cor:Pi of lattice basis is saturated
  Horn} reveals that for sufficiently generic $p\in X$,  
there is an isomorphism of solution spaces for lattice basis binomial
and Horn hypergeometric $D$-modules:  
\[
\hspace*{3.65cm}
\Sol_p\left(\frac{D_\barX}{I(B)+\<E-A\kappa\>}\right)
\cong 
\Sol_{p^B}\left(\frac{D_\barZ}{\sHorn(B,\kappa)}\right).
\hspace{3.35cm}
\hexagon
\] 
\end{remark}

\subsection{Characteristic varieties and singular loci}
\label{subsec:E-beta:charVar Sing}

The homomorphism $\dkap$ can be used to explain the image of the
characteristic variety and singular locus of $M$ under $\Delta^{\tilde A}_B$ when $M$ is in $\Tirr(X,\beta)$. 

\begin{proposition}
\label{prop:charVar X}
Let $M$ be a module in $\Tirr(X,\beta)$. 
If $B$ has columns that $\ZZ$-span $\ker_{\ZZ}(A)$, then 
the characteristic variety and singular locus of $\Pi^{\tilde A}_B(M)$ intersected with $Z$  
are geometric quotients of the intersection of the characteristic variety 
and singular locus of $M$ with $X$, respectively.
\end{proposition}

\begin{proof}
Note first that taking $F$-associated graded objects commutes with
taking invariants. Also, taking invariants produces a $[D_X]^T$-module $[M]^T$. 
Moreover, by the definition of $\AEmaxmods{X,\beta}$ and 
Proposition~\ref{prop:dkap ker}, we see that $[M]^T$ is already naturally a 
$D_Z$-module that is isomorphic to $\Delta^{\tilde A}_B(M)$. 
We now have the result, since taking torus invariants 
induces categorial quotients, which in this case separates orbits.
\end{proof}

Combining Propositions~\ref{prop:characteristic variety Pi} 
and~\ref{prop:charVar X}, we can compute the characteristic variety
and singular locus of $\Pi^{\tilde A}_B(M)$ in terms of those for 
$M$ under certain assumptions. 

\begin{corollary}
\label{cor:charVar variety Pi E-beta}
Let $M$ be a nonzero regular holonomic module in $\Tirr(\barX,\beta)$. 
If $B$ has columns that $\ZZ$-span $\ker_{\ZZ}(A)$, 
then $\charVar(\Pi^{\tilde A}_B(M))$ and $\Sing{\Pi^{\tilde A}_B(M)}$
are the unions of the geometric quotient of  
$\charVar(D_X\otimes_{D_\barX}M)$ and $\Sing{D_X\otimes_{D_\barX}M}$
with, respectively, the coordinate hyperplane conormals in $\barZ$ and
the coordinate hyperplanes in $\barZ$. 
\qed
\end{corollary}


\section*{{\bf Part II: \ Binomial \texorpdfstring{$D$}{D}-modules and hypergeometric systems}}

\section{Binomial \texorpdfstring{$D$}{D}-modules}
\label{sec:binomial}

In this section, we summarize results from~\cite{DMM/D-mod,binomial-slopes,BMW13-holSing} regarding the holonomicity and regularity of binomial $D$-modules.  

\subsection{Overview}
\label{subsec:binomialHol}

The holonomicity of a binomial $D$-module is controlled by the
primary decomposition of the underlying binomial ideal. Primary
decomposition of binomial ideals has several special
features. Eisenbud and Sturmfels have shown that the associated primes
of a binomial ideal are binomial ideals themselves and that the
primary decomposition of a binomial ideal can be chosen to be
binomial~\cite{eisenbud-sturmfels}. A more detailed study of the
primary components of a binomial ideal, geared towards applications to 
binomial $D$-modules, appears in~\cite{DMM/primdec}. 

Binomial ideals can
have two types of primary components, \emph{toral} and \emph{Andean},
according to how they behave with respect to the inherited torus
action. We recall from \cite{eisenbud-sturmfels} that if 
$I\subseteq \CC[\del_x]$ is a binomial ideal, then every
associated prime of $I$ is of the form $I'+\<\dx{i} \mid i \notin
\sigma \>$, where $\sigma \subseteq \{1,\dots,n\}$, $I'$ is generated
by elements of $\CC[\del_\sigma] \defeq \CC[\dx{i} \,|\, i \in \sigma]$, and
the intersection $I' \cap \CC[\del_{\sigma}]$ is a toric ideal after
rescaling the variables in $\CC[\del_{\sigma}]$.

\begin{definition}
\label{def:toral-andean}
Let $I$ be an $A$-graded binomial ideal in $\CC[\del_x]$, and let
$\mathfrak{p} = I' + \< \dx{i} \mid i \notin \sigma\>$ be an
associated prime of $I$ as above, where we have rescaled the variables
so that $I' \cap \CC[\del_\sigma]$ is a toric ideal. Denote by
$A_{\sigma}$ the submatrix  
of $A$ consisting of the columns in $\sigma$. 
If $I' \cap \CC[\del_{\sigma}] = I_{A_\sigma}$, then $\mathfrak{p}$
is called a \emph{toral} associated prime of $I$, and the
corresponding primary component $\mathfrak{p}$ is also called toral. An associated
prime (or primary component) that is not toral is called \emph{Andean}.
\end{definition}

It is the Andean components of a binomial ideal that cause failure of
holonomicity for the corresponding binomial $D_\barX$-module. 
To make this precise, 
let $V$ be an $A$-graded $\CC[\del_x]$-module. The set of
\emph{quasidegrees} of $V$ is 
$\qdeg(V) \defeq \overline{\{ \alpha \in \ZZ^d \subseteq \CC^d \mid V_{\alpha} \neq 0 \}
}^{\text{Zariski}}$, 
where the closure is taken with respect to the Zariski topology in $\CC^d$.
If $I$ is an $A$-graded binomial $\CC[\del_x]$-ideal, then 
the set $\bigcup\qdeg(\CC[\del_x]/\sC)$, 
where the union runs over the Andean
components $\sC$ of $I$, is called the \emph{Andean arrangement of $I$}.

The following theorem collects results about binomial $D_\barX$-modules. 
Except for items~\eqref{item:binomial:Lhol} 
and~\eqref{item:binomial:regHol},
which are from~\cite{BMW13-holSing} and~\cite{binomial-slopes} 
respectively, all of these facts are proved in~\cite{DMM/D-mod}.

\begin{theorem}\cite{DMM/D-mod,binomial-slopes,BMW13-holSing}
\label{thm:main result for binomial D-mods}
Let $I \subseteq \CC[\del_x]$ be an $A$-graded binomial ideal, and
consider the binomial $D_\barX$-module $M = D_\barX/(I+\<E-\beta\>)$. 
\begin{romanlist}
\vspace{-2mm}
\item\label{item:binomial:hol:arr}
	The module $M$ is holonomic if and only 
	if $\beta$ does not lie in the Andean 
	arrangement of $I$. 
\item\label{item:binomial:hol:count}
	The 
	module $M$ is holonomic if and only if
        its holonomic rank is countable. 
\item\label{item:binomial:hol:finite} 
	The 
	module $M$ is holonomic if and only if
        its holonomic rank is finite. 
\item\label{item:binomial:Lhol} 
If $M$ fails to be holonomic, then $\charVar(M)$ has a component 
in $T^*X$ of dimension more than $n$. 
\item\label{item:binomial:regHol}
	Assume that $\beta$ does not lie on the 
	Andean arrangement of $I$. Consider the 
	set 
	\begin{equation}
	\label{eqn:important components}
		\{ \sC \text{ toral primary component of } 
		I \mid \beta \in \qdeg(\sC)\}.
	\end{equation}
	Then $M$ is regular holonomic if and
    only if each ideal in~\eqref{eqn:important components} is 
    standard $\ZZ$-graded. 
\item\label{item:binomial:Horn-regHol}
	A lattice basis binomial $D_\barX$-module 
	associated to a
    matrix $B$ is regular holonomic if and 
    only if it is holonomic
    and the ideal $I(B)$ is standard 
    $\ZZ$-graded. 
\item\label{item:binomial:rankJumps}
	The holonomic rank of $M$ is 
	minimal as a
	function of $\beta$ if and only if 
	$\beta$ does not belong to
    the union of the Andean arrangement of $I$ with  
	$\qdeg\left(\bigoplus_{i<d} 
	H_{\mathfrak{m}}^i\left(\CC[\del_x]/
		I_{\text{toral}}\right)\right)$, 
	where $H_{\mathfrak{m}}^i(-)$ denotes local cohomology with
        support $\mathfrak{m}=\<\dx{1},\dots,\dx{n}\>$ 
	and $I_{\text{toral}}$ is the intersection of the toral primary 
	components of $I$. 
\item\label{item:binomial:minRank}
	Suppose that the Andean arrangement of $I$ is not all of
        $\CC^d$. For each toral associated prime of $I$, let $A_{\sigma}$ 
	be the matrix from Definition~\ref{def:toral-andean}, and 
	denote by $\vol(A_\sigma)$ the lattice volume of the polytope 
	$\conv(0,A_{\sigma})$ in the lattice $\ZZ A_\sigma$. 
	The minimal rank attained by $M$ is the sum
	over the toral primary components of $I$ of 
	${\rm mult}(I,\sigma) \cdot \vol(A_\sigma)$, where
	${\rm mult}(I,\sigma)$ is the multiplicity of the 
	associated prime $I_{A_\sigma} + 
	\<\dx{j} \mid j \notin \sigma\>$ in $I$. 
\item\label{item:binomial:directSum}
	Define a module $P_I$ by the exact sequence
\[
0 \longrightarrow \frac{\CC[\del_x]}{I} \longrightarrow \bigoplus
\frac{\CC[\del_x]}{\sC} \longrightarrow P_I \longrightarrow 0
\] 
	where the direct sum is over the 
        primary components $\sC$ of $I$.  
	If $\beta$ lies outside the union of $\qdeg(P_I)$ with the
        Andean arrangement of $I$, then $M$ is isomorphic to the
        direct sum over 
        the modules $D_\barX/(\sC+\<E-\beta\>)$, where $\sC$ lies in~\eqref{eqn:important components}. 
	\qed
\end{romanlist}
\end{theorem}

The Andean arrangement and all other quasidegree sets in
Theorem~\ref{thm:main result for binomial D-mods} are unions of
finitely many integer translates of subspaces of the form $\CC
A_\sigma$, where $\sigma \subseteq \{1,\dots,n\}$. If the Andean
arrangement of $I$ is not all of $\CC^d$, then the other quasidegree
sets are also proper subsets of $\CC^d$. 

The holonomicity of an $A$-hypergeometric system was first shown in~\cite{GGZ,adolphson}, and information on the rank of these systems can be found in~\cite{MMW,berkesch}. 

\subsection{Regular holonomicity and Nilsson series}
\label{subsec:DMMregNilsson}

We now discuss Nilsson solutions of binomial $D$-modules, generalizing
statements in~\cite{nilsson} for $A$-hypergeometric systems.  
We will make use of a homogenization operation. In this direction, set 
\[ 
{\rho(A)} \defeq 
\left[ \begin{array}{c|c}
1 & 1\  \cdots \ 1 \\ \hline
0 & \\
\vdots & A \\
0 & 
\end{array} \right]. 
\] 
Given an $A$-graded binomial ideal $I\subseteq \CC[\del_x]$, let
$\rho(I)$ denote the homogenization of $I$ with respect to an
additional variable $\dx{0}$.  
Here, $\dx{0}$ corresponds to a variable $x_0$ giving rise to
coordinates $(x_0,x_1,\dots,x_n)$ on $\breve{X} \defeq \CC^{n+1}$.  
Note that $\rho(I)$ is $\rho(A)$-graded. 
For a fixed $\beta_0\in\CC$ and $\beta\in\CC^d$, write $E_\rho -
(\beta_0,\beta)$ for the sequence of $d+1$ Euler operators associated
to $\rho(A)$ and the vector $(\beta_0,\beta)$. We define the
\emph{homogenization} of the binomial $D$-module  
$M = D_\barX/(I+\<E-\beta\>)$ to be 
\[
\rho(M,\beta_0) \defeq \frac{D_{\breve{X}}}{\rho(I)+\<E_\rho - (\beta_0,\beta)\>}. 
\]

A vector $w\in\RR^n$ is a \emph{generic binomial weight vector} for
a binomial $D$-module $M = D_\barX/(I+\<E-\beta\>)$, if it satisfies the
conditions from Definition~\ref{def:generic:weight} 
and also, for all $w'\in\Sigma$, we have $\gr^w(I) = \gr^{w'}(I)$. 
We say that $w\in\RR^n$ is a \emph{perturbation} of $w_0\in\RR^n_{>0}$
if there exists an open cone $\Sigma$ as in the previous definition,
with $w\in \Sigma$ such that $w_0$ lies in the closure of $\Sigma$.  
Our proofs also make use of the homogenization of a basic Nilsson
solution  $\phi = \sum_{u\in C}x^{v+u}p_u(\log(x))$ of $M$ in the direction of a binomial weight vector $w\in\RR^n$ for $M$. 
From the definition of $\rho(\phi)$ in~\cite[\S3]{nilsson}, 
it follows as in~\cite[Proposition~3.17]{nilsson} that if $|u|\geq 0$ 
for all $u\in C$, then $\rho(\phi)$ is a basic 
Nilsson solution of $\rho(M,\beta_0)$ in the direction of $(0,w)$.

\begin{proposition}
\label{prop:binomial rho preserve reg hol}
Let $I\subseteq\CC[\del_x]$ be an $A$-graded binomial ideal, and set
$M = D_\barX/(I+\<E-\beta\>)$.
Suppose that $\beta_0\in\CC$
is generic, in that, if $J$ is the intersection of the $\sC$ in
\eqref{eqn:important components}, 
then $(\beta_0,\beta)$ is not in the
quasidegree set of $P_{\rho(J)}$ from Theorem~\ref{thm:main result for binomial D-mods}.\eqref{item:binomial:directSum}.   
If $\beta$ is not in the Andean arrangement of $I$, then the binomial
$D$-module $\rho(M,\beta_0)$ is regular holonomic. 
\end{proposition}

\begin{proof}
Since $\beta$ is not in the Andean arrangement of $I$, we may assume
that all primary components of $I$ are toral.  
Since $\beta_0$ is generic and $\rho$ commutes with taking primary decompositions, $\rho(I)$ also has only toral primary
components.  
Further, 
by Theorem~\ref{thm:main result for binomial D-mods}.\eqref{item:binomial:directSum}, 
$\rho(M,\beta_0)$ is isomorphic to 
$\bigoplus D_{\breve{X}}/(\rho(\sC)+\<E_\rho-(\beta_0,\beta)\>)$, 
where the direct sum is over the toral primary components $\sC$ of $I$.  
Since each ideal $\rho(\sC)$ is $\ZZ$-graded, the result now follows from 
Theorem~\ref{thm:main result for binomial D-mods}.\eqref{item:binomial:regHol}.
\end{proof}

We now generalize~\cite[Theorem~7.3]{christine-thesis}, which was first stated for 
$A$-hypergeometric systems. 

\begin{proposition}
\label{prop:binomial homog rank preserved}
Let $I\subseteq\CC[\del_x]$ be an $A$-graded binomial ideal, and set $M = D_\barX/(I+\<E-\beta\>)$. 
If $\beta$ does not lie in the Andean arrangement of $I$ and
$\beta_0\in\CC$ is generic as in Proposition~\ref{prop:binomial rho
  preserve reg hol}, then  
\[
\rank(M) = \rank(\rho(M,\beta_0)).
\]
\end{proposition}
\begin{proof} 
We may assume that $I$ is equal to the intersection of the $\sC$ in
\eqref{eqn:important components}.  
By Theorem~\ref{thm:main result for binomial
  D-mods}.\eqref{item:binomial:hol:arr} and the genericity of
$\beta_0$, $M$ and $\rho(M,\beta_0)$ are holonomic and thus of finite
rank.  
Given the collection $\{\sC\}$ of toral primary components of $I$, we
may extend $0\to \CC[\del_x]/I \to \bigoplus \CC[\del_x]/\sC$ to a
\emph{primary resolution} of $\CC[\del_x]/I$,
see~\cite[\S4]{monomialGKZ}.  
To compute the rank of $M$, we then apply Euler--Koszul homology to
this resolution and follow the resulting spectral sequence, as in the
proof of~\cite[Theorem~4.5]{monomialGKZ}.  
That argument and the fact that the theorem has been proven for
$A$-hypergeometric systems in~\cite[Theorem~7.3]{christine-thesis}
imply that this procedure and its associated numerics are compatible
with $\rho$,  
and this compatibility yields the desired result. 
\end{proof}

With the definitions of homogenization and 
Propositions~\ref{prop:binomial rho preserve reg hol}
and~\ref{prop:binomial homog rank preserved} in hand, the necessary
ingredients are in place to apply the original proof 
of~\cite[Theorem~6.4]{nilsson} to arbitrary holonomic binomial $D$-modules. 
This result, appearing now in Theorem~\ref{thm:nilsson binomial
  enough}, provides a method to compute the rank of these modules via
certain Nilsson solutions.   
Note also that the statement runs in parallel to 
Theorem~\ref{thm:solutions-of-regular-holonomic}, without requiring
regular holonomicity.  

\begin{theorem}
\label{thm:nilsson binomial enough}
Let $I\subseteq\CC[\del_x]$ be an $A$-graded binomial ideal, and set
$M = D_\barX/(I+\<E-\beta\>)$.  
Assume that $\beta$ does not lie in the Andean arrangement of $I$. 
If $w \in \RR^n$ is a generic binomial weight vector of $M$ that is a
perturbation of $\boldone_n$, then   
\[
\dim_\CC ( \sN_w(M) ) = \rank(M).
\] 
Further, there exists an open set $U\subseteq\barX$ such that the
basic Nilsson solutions of $M$ in the direction of $w$ simultaneously
converge at each $p\in U$, and as such, they form a basis for
$\Sol_p(M)$.  
\qed
\end{theorem}

Finally, we generalize~\cite[Proposition~7.4]{nilsson} in
Theorem~\ref{thm:nilsson binomial less}.  
Combining this result with
Theorem~\ref{thm:solutions-of-regular-holonomic}, we obtain a new   
criterion for the regular holonomicity of binomial $D$-modules. 
In particular, instead of considering a family of derived solutions,
regularity of a binomial $D$-module can be tested through a single
collection of Nilsson solutions.  

\begin{theorem}
\label{thm:nilsson binomial less}
Let $I\subseteq\CC[\del_x]$ be an $A$-graded binomial ideal such that
$M = D_\barX/(I+\<E-\beta\>)$ is not regular holonomic. If $\beta$
does not lie in the Andean arrangement of $I$,  
then there exists a generic binomial weight vector $w \in \RR^n$ of $M$ such that 
$\dim_\CC ( \sN_w(M) ) < \rank(M)$. 
\end{theorem}

\begin{proof}
The original proof of~\cite[Proposition~7.4]{nilsson} can be applied, 
since $\rho$ is commutes with taking primary decompositions. Note that
Propositions~\ref{prop:binomial rho preserve reg hol}
and~\ref{prop:binomial homog rank preserved} provide the binomial generalizations of the originally toric results needed in this
proof.  
\end{proof}

\subsection{Irreducible monodromy representation} 
\label{subsec:binomialMonodromy}

We now characterize when a binomial $D$-module has reducible monodromy representation.
To achieve this, we extend the results \cite[Theorems~3.1 and~3.2]{sw-irred} (see
also \cite{beukers-irred,saito-irred}) on $A$-hypergeometric systems to arbitrary binomial $D$-modules. 

Given $\sigma\subseteq\{1,\dots,n\}$, let $A_\sigma$ denote
the submatrix of $A$ given by the columns of $A$ that lie in $\sigma$, and 
let $I_{A_\sigma}\subseteq\CC[\dx{i}\mid i\in\sigma]$ denote the toric
ideal given by $A_\sigma$ in the variables corresponding to $\sigma$.  

A \emph{face} $G$ of $A$, denoted $G\preceq A$, is a subset $G$ of the
column set of $A$ such that there is a linear functional $\phi_G\colon
\ZZ A\to \ZZ$ that vanishes on $G$ and is positive on any element of
$A\minus G$.  
If $G\preceq A$, then the parameter $\beta\in\CC^d$ is
\emph{$G$-resonant} if $\beta\in\ZZ A+\CC G$.  
If $\beta$ is $H$-resonant for all faces $H$ properly containing $G$,
but not for $G$ itself, then $G$ is called a \emph{resonance center
  for $\beta$}.  
It is said that $A$ is a \emph{pyramid} over a face $G$ if 
$\vol_{\ZZ G}(G) = \vol_{\ZZ A}(A)$, where volume is a normalized (or
simplicial) volume computed with respect to the given ambient
lattices.  

\begin{theorem}
\label{thm:binomial irred}
Let $I$ be an $A$-graded binomial ideal in $\CC[\del_x]$ and suppose
that $\beta\in\CC^d$ does not lie in the Andean arrangement of
$I$. The binomial $D_\barX$-module $D_\barX/( I+\<E-\beta\>)$ has
irreducible monodromy representation if and only if the following 
conditions hold: 
\begin{enumerate}
\vspace{-2mm}
\item 
\label{item:monodromy:1}
For some $\sigma\subseteq \{1,\dots,n\}$, the intersection of the
toral components $\sC$ of $I$ for which $\beta\in\qdeg(\sC)$  
equals  $I_{A_\sigma} + \<\dx{i}\mid i\notin\sigma\>$. 
\item  
\label{item:monodromy:2}
For any $\sigma$ as in~\eqref{item:monodromy:1}, there is a unique
resonance center $G\preceq A_\sigma$ of $\beta$, and $A_\sigma$ is a pyramid over $G$. 
\end{enumerate}
\end{theorem}
\begin{proof} 
We may assume that $I$ is equal to the intersection of the toral
components $\sC$ of $I$ for which $\beta\in\qdeg(\sC)$.  
If $I$ is not prime, pick an associated prime of $I$. After possibly
rescaling the variables, since $\beta$ does not lie in the Andean
arrangement of $I$, such an 
associated prime is of the form $I_{A_\sigma}+\<\dx{i} \mid i \notin
\sigma\>$ for some $\sigma\subseteq \{1,\dots, n\}$. Let 
$N = D_\barX/( I_{A_\sigma}+\<\dx{i} \mid i \notin
\sigma\>+\<E-\beta\>)$. Then $N$ is a submodule of 
$M$ such that $\CC(x) \otimes_{\CC[x]} N$ is a nonzero
proper submodule of $\CC(x) \otimes_{\CC[x]} M$. 
Thus $M$ always has reducible monodromy representation in this case.

We may now assume that $I$ is prime. Then since $\beta$ does not lie
in the Andean arrangement of $I$, there exists a subset $\sigma
\subseteq \{1,\dots,n\}$ 
such that $I=I_{A_\sigma}+\<\dx{i} \mid i \notin \sigma\>$, after
possibly rescaling the variables. 
Since the pyramid condition implies the uniqueness
of a resonance center $G$ of $\beta$ by~\cite[Lemma~2.9]{sw-irred}, 
we have reduced the proof of the statement for binomial $D$-modules to
the $A$-hypergeometric setting, which was already proven  
in~\cite[Theorems~3.1 and~3.2]{sw-irred}. 
\end{proof}

\section{Torus invariants and binomial \texorpdfstring{$D_\barX$}{DbarX}-modules}
\label{sec:binomial hol Pi}

In this section, we strengthen our results on the transfer of (regular) holonomicity 
through $\Pi^{\tilde A}_B$ from Theorem~\ref{thm:transfer theorem},
still following Convention~\ref{conv:BtildeACK}, in the case that the input is a binomial $D_\barX$-module. 
Remark~\ref{rem:charVar Pi binomial} provides an explicit description
of the characteristic variety 
of $\Pi^{\tilde A}_B$ applied to a regular holonomic binomial
$D_\barX$-module.  

\begin{theorem}
\label{thm:binomial quotients}
Let $I \subseteq \CC[\del_x]$ be an 
$A$-graded binomial ideal. 
Then the following hold for the binomial $D_\barX$-module $M = D_\barX/(I+\<E-\beta\>)$. 
\vspace{-2mm}
\begin{enumerate}
\item 
\label{thm:binomial quotients:Lhol}
The binomial $D_\barX$-module $M$ is holonomic if and only if
  $\Pi^{\tilde A}_B(M)$ is holonomic.  
\item 
\label{thm:binomial quotients:rank}
  If $M$ is holonomic, then the rank of $\Pi^{\tilde A}_B(M)$ is equal
  to $[\ker_{\ZZ}(A):\ZZ B]\cdot\rank(M)$.  
\item 
\label{thm:binomial quotients:reg hol}
  The module $M$ is regular holonomic if and only if $\Pi^{\tilde
  A}_B(M)$ is regular holonomic.  
\end{enumerate}
\end{theorem}

\begin{corollary}
\label{cor:binomial quotients}
Let $I \subseteq \CC[\del_x]$ be an $A$-graded binomial ideal. 
The statements of Theorem~\ref{thm:main result for binomial D-mods}
hold if we replace the binomial $D_\barX$-module $M = D_\barX/( I+\<E-\beta\>)$ by  
$\Pi^{\tilde A}_B(M)$. 
Minor modifications are needed; for clarity, we explain the changes in four items. 
\vspace{-2mm}
\begin{enumerate}
\item[(\emph{\ref{item:binomial:Lhol}})] 
Replace $T^*X$ and $n$ respectively by $T^*Z$ and $m$. 
\item[(\emph{\ref{item:binomial:Horn-regHol}})]
  The saturated Horn $D_\barZ$-module $D_\barZ/\sHorn(B,\kappa)$ is
  (regular) holonomic if and only if the binomial $D_\barX$-module
  $D_\barX/(I(B)+\<E-A\kappa\>)$ is (regular) holonomic. In
  particular, it follows that $D_\barZ/\sHorn(B,\kappa)$ is regular
  holonomic if and only if it is holonomic and the rows of the matrix
  $B$ sum to $\boldzero_m$. 
\item[(\emph{\ref{item:binomial:minRank}})]
  The minimal rank value above must be scaled by $[\ker_{\ZZ}(A):\ZZ B]$. 
\item[(\emph{\ref{item:binomial:directSum}})] 
  If $\beta$ does not lie in the union of $\qdeg(P_I)$ with the Andean
  arrangement of $I$, then  $\Pi^{\tilde A}_B(M)$ is isomorphic to the
  direct sum over the toral primary components $\sC$ of $I$ of the
  modules $\Pi^{\tilde A}_B(D_\barX/(\sC+\<E-\beta\>))$, where $\sC$ lies in~\eqref{eqn:important components}. 
\end{enumerate}
\end{corollary}
\begin{proof}
Only~\eqref{item:binomial:Horn-regHol} does not follow immediately
from Theorem~\ref{thm:binomial quotients}. For this, we must also use
the decomposition of $\Pi^{\tilde A}_B(D_\barX/(I(B)+\<E-A\kappa\>))$
from Corollary~\ref{cor:Pi of lattice basis is saturated Horn}.  
Note that each summand in this decomposition is actually isomorphic to
the first summand, $D_\barZ/\sHorn(B,\kappa)$.  
The parameter shifts $\kappa+Ck$ in the remaining summands are a red
herring, thanks to the graded shifts induced by multiplication by
$z^{k/\varkappa}$.  
\end{proof}

\begin{proof}[Proof of Theorem~\ref{thm:binomial quotients}] 

The only if direction of item~\eqref{thm:binomial quotients:Lhol} 
follows from Theorem~\ref{thm:transfer theorem}. For the converse, apply Theorem~\ref{thm:main result for binomial D-mods}.\eqref{item:binomial:Lhol} and Proposition~\ref{prop:characteristic variety Pi}. 

To prove~\eqref{thm:binomial quotients:rank}, since $M$ is holonomic,   
by Theorem~\ref{thm:nilsson binomial enough} 
there is a generic binomial weight vector $w\in\RR^n$ such that 
at a nonsingular point $p$, the basic
Nilsson solutions of $M$ in the direction of $w$ span $\Sol_p(M)$.  
Thus Theorem~\ref{thm:solutions E} implies~\eqref{thm:binomial quotients:rank}.  

The forward implication of~\eqref{thm:binomial quotients:reg hol} is
Theorem~\ref{thm:transfer theorem}.\eqref{item:regular-hol barX}, so
it remains to establish the converse in the binomial setting.  
Suppose that $M$ is not regular holonomic. 
Then by Theorem~\ref{thm:nilsson binomial less}, 
there exists a generic weight vector $w\in\RR^n$ for $M$ such that 
$\dim_\CC\sN_w(M)<\rank(M)$. 
Thus by Theorem~\ref{thm:solutions E} applied to these Nilsson
solutions, we see that, with $\ell = w B$,  
\[
\dim_\CC\sN_\ell(\Pi^{\tilde A}_B(M)) 
= [\ker_{\ZZ}(A):\ZZ B]\cdot\dim_\CC\sN_w(M).
\] 
By slight perturbation of $w$, if necessary, $\ell$ will be a generic
weight vector for $\Pi^{\tilde A}_B(M)$.  
But then by~\eqref{thm:binomial quotients:rank}, we see that 
$\dim_\CC\sN_\ell(\Pi^{\tilde A}_B(M)) < 
[\ker_{\ZZ}(A):\ZZ B]\cdot \rank(M) = \rank(\Pi^{\tilde A}_B(M))$. 
Thus  
Theorem~\ref{thm:solutions-of-regular-holonomic} implies that 
$\Pi^{\tilde A}_B(M)$ cannot be regular holonomic. 
\end{proof}

\begin{remark}
\label{rem:charVar Pi binomial}
Let $I$ be an $A$-graded binomial $\CC[\del_x]$-ideal, and assume that $\beta
\in \CC^d$ lies outside of the Andean arrangement of $I$, so that $M =
D_\barX/( I+\<E-\beta\>)$ is holonomic. 
Theorem~4.3 in~\cite{binomial-slopes} states that $\charVar(M)$ is equal to the 
union of the characteristic varieties of the binomial
$D_\barX$-modules corresponding to associated primes of $I$ whose 
components belong to~\eqref{eqn:important components}.
Since prime binomial ideals are isomorphic to toric ideals, 
it suffices 
to compute the characteristic
varieties of $A$-hypergeometric systems, which is done explicitly by
Schulze and Walther in~\cite{schulze-walther-duke}. 
Therefore, combining~\cite{binomial-slopes,schulze-walther-duke} with
Proposition~\ref{prop:charVar X} yields a description of 
the characteristic variety of $\Delta^{\tilde A}_B(D_X/(
I+\<E-\beta\>))$. If $M$ is regular holonomic, then  
Proposition~\ref{prop:characteristic variety Pi} provides an explicit
description of $\charVar(\Pi^{\tilde A}_B(M))$. 
\endrk
\end{remark}

\begin{example}
\label{ex:Horn vs sHorn charVar}
If $A = [1\ 2\ 3]$ and 
$B = \begin{bmatrix} -2 & \phantom{-}1 \\ \phantom{-}1 & -2 \\
  \phantom{-}0 & \phantom{-}1 \end{bmatrix}$,  
then when $\kappa = [0,0,0]$, 
\begin{align*}
\charVar\left(
	\frac{D_\barX}{ I(B)+\<E-A\kappa\>} 
	\right)
&= \Var(\<\xi_1,\xi_2,\xi_3\>) 
  \cup \Var(\<\xi_1,\xi_2,x_3\>)
\intertext{
by computation in \texttt{Macaulay2}~\cite{M2}. 
Thus by Proposition~\ref{prop:characteristic variety Pi}, }
\charVar\left(
\frac{D_\barZ}{\sHorn(B,\kappa)}
\right) 
&=
\Var(\< z_1\zeta_1,z_2\zeta_2 \>). 
\intertext{On the other hand, \texttt{Macaulay2}~\cite{M2} reveals that}
\charVar\left(
\frac{D_\barZ}{\Horn(B,\kappa)}
\right) 
&= 
\Var(\< z_1\zeta_1,z_2\zeta_2 \>) \cup \Var(\< z_1,4z_2-1\>), 
\end{align*}
so the component $\Var(\< z_1,4z_2-1\>)$ of
$\charVar(D_\barZ/\Horn(B,\kappa))$ cannot be obtained from those of 
$D_\barX/( I(B)+\<E-A\kappa\>)$.  
In particular, the holonomicity of $D_\barX/( I(B)+\<E-A\kappa\>)$
does not by itself guarantee the holonomicity of
$D_\barZ/\Horn(B,\kappa)$. However, we show in \cite[Theorem~3.5]{bmw-Horn} that this is the case under strong conditions on
$\beta$.   
In this example, this assumption is equivalent to lying outside of the
Andean arrangement of $A$.  
\endrk
\end{example}

\begin{example}
\label{ex:preliminaryCounterexample}
Consider the matrices 
\[
B = \left[ \begin{array}{rrr}
1  &  1 & 2 \\
-1 & -1 & 0 \\
0  &  0 & -1 \\
1  & 0  & 0 \\
0  & 1  & 0 \\
-1 & 0  &0 \\
0  &-1 & 0 
\end{array} \right]  , \;\;
B' = \left[ \begin{array}{rrr}
1  &  1 & -2 \\
-1 & -1 & 0 \\
0  &  0 & 1 \\
1  & 0  & 0 \\
0  & 1  & 0 \\
-1 & 0  &0 \\
0  &-1 & 0 
\end{array} \right]  , \;\;
\text{and}\;\; \kappa= \begin{bmatrix} 2 \\ 0 \\ 0 \\ 0 \\ 0 \\ 0 \\
  0 \end{bmatrix},
\]
noting that $B'$ is obtained by $B$ by changing the sign of the last column; in particular, they correspond to the same torus action. 
However, they do not yield Horn hypergeometric systems with the same behavior: $D_\barZ/\Horn(B',\kappa)$ is holonomic, while $D_\barZ/\Horn(B,\kappa)$ is not, since its characteristic variety has a component associated to the ideal $\<z_3,z_1\zeta_1+z_2\zeta_2\>$. 
\endrk
\end{example}

We conclude this section by addressing the irreducibility of 
monodromy representation of Horn $D_\barZ$-modules. 
Note first that Theorem~\ref{thm:transfer theorem} and
Theorem~\ref{thm:binomial irred} together provide a test for the
reducibility of monodromy representation for the image of a binomial
$D_\barX$-module under $\Pi^{\tilde A}_B$.  

\begin{corollary}
\label{cor:monodromy horn}
If $\beta = A\kappa$ does not lie in the Andean arrangement of $I(B)$, 
then the Horn $D_\barZ$-modules $D_\barZ/\Horn(B,\kappa)$ and $D_\barZ/\sHorn(B,\kappa)$ 
have irreducible monodromy representation if and only if  
(1) and (2) in Theorem~\ref{thm:binomial irred} hold for $I = I(B)$. 
\end{corollary}
\begin{proof}
Note that the three incarnations of Horn systems are isomorphic after
tensoring with $\CC(z)$, so it is enough to consider
$D_\barZ/\sHorn(B,\kappa)$.  
The result thus follows from combining Theorem~\ref{thm:binomial
  irred} with Corollary~\ref{cor:Pi of lattice basis is saturated
  Horn} and Remark~\ref{rem:solutions Horn}. 
\end{proof}

\section{A sufficient condition for holonomicity of \texorpdfstring{$\Horn(B,\kappa)$}{Horn(B,kappa)}}
\label{sec:hornHolonomicity}

In this section, we provide a sufficient condition for the holonomicity of
$\Horn(B,\kappa)$. If $\Horn(B,\kappa)$ is holonomic, 
then $\sHorn(B,\kappa)$ is holonomic. However, the converse is not true
by Example~\ref{ex:preliminaryCounterexample}; one needs a stronger
assumption on $\sHorn(B,\kappa)$.

Example~\ref{ex:preliminaryCounterexample}
involves three variables. This is not a coincidence. The main result in this section,
Theorem~\ref{thm:hornHolo}, implies
Corollary~\ref{coro:bivariateHolonomic}, which states that when $B$
has fewer than three columns,
$\Horn(B,\kappa)$ is holonomic if and only if $\sHorn(B,\kappa)$ is
holonomic.

\begin{definition}
If $\varnothing\neq\gamma\subseteq [m]=\{1,\dots,m\}$, denote by $B[\gamma]$ the matrix
whose columns are the columns of $B$ indexed by $\gamma$.
A parameter $\kappa \in \CC^n$ is \emph{toral for $B$} if
$\sHorn(B,\kappa)$ is holonomic, and it is \emph{completely
  toral for $B$} if $\sHorn(B[\gamma],\kappa)$ is holonomic for every
$\varnothing \neq \gamma \subseteq [m]$.
\endrk
\end{definition}

\begin{lemma}
The set of (completely) toral parameters of $B$ is Zariski
open in $\CC^n$. 
\end{lemma}

\begin{proof}
It is enough to show this for the toral parameters of $B$. For $A$ as in Convention~\ref{conv:A}, the set $\{\beta = A\kappa \mid
\sHorn(B,\kappa) \text{ is not holonomic}\} \subseteq \CC^d$ is Zariski closed by
Theorem~\ref{thm:main result for binomial D-mods}.\eqref{item:binomial:hol:arr} and
Corollary~\ref{cor:binomial
  quotients}.\emph{\eqref{item:binomial:Horn-regHol}}. Thus 
the set of nontoral parameters for $B$ is closed in $\CC^n$. 
\end{proof}

\begin{theorem}
\label{thm:hornHolo}
If $\kappa \in \CC^n$ is completely toral for $B$, then
$\Horn(B,\kappa)$ is holonomic. If there exists a completely toral
parameter for $B$, then $\Horn(B,\kappa)$ is holonomic for generic parameters.
\end{theorem}

\begin{proof}
Let $\kappa$ be a completely toral parameter for $B$. By Lemma~\ref{lem:TcrossZ-fg},
the completely toral parameters are then Zariski-dense, which implies
the last claim.
Let $C\defeq \charVar\left(D_\barZ/\Horn(B,\kappa)\right)$, and
consider the stratification of $\barZ$ by the multiplicity of
the function $z_1\cdots z_m$. The (relatively) open strata are in
bijection with the subsets $\gamma$ of $[m]$, given by
$Z_\gamma\defeq (\CC^*)^\gamma\times\{0\}^{[m]\minus \gamma}$,
the set of points in $\barZ = \CC^m$ whose entries indexed by $\gamma$
are nonzero, and whose entries indexed by $[m]\minus\gamma$ are
zero. There is a corresponding stratification of $T^*\barZ$ whose open
strata are $T^*_\gamma\defeq \left((\CC^*)^\gamma\times\{0\}^{[m]\minus
  \gamma}\right)\times \CC^m$.

It suffices to show that the intersection of $C$ with $T_\gamma$ is
$m$-dimensional for each $\gamma$. This is obvious for
$\gamma=\varnothing$, and the case $\gamma=[m]$ follows because
$\Horn(B,\kappa)$ and $\sHorn(B,\kappa)$ agree on the open subset
$Z=Z_{[m]}$. Let $\gamma\subseteq [m]$ be any other choice.
Let $D_\gamma$ be the Weyl algebra on $\CC^\gamma$, and let $P=\sum_{
  k\in \gamma} P_k(q_k-z_kp_k)$ be an element of $\Horn(B,\kappa)$
where all $P_k\in D_\gamma$. The symbol of such operator is one of the
equations that cut out $C$. When restricted to $T_\gamma$, it becomes
the (pullback under the projection $T^*_\gamma\to T^*\CC^\gamma$ of the) symbol of
$P_\gamma=\sum_{ k\in \gamma}
P_k(q^\gamma_k-z_kp^\gamma_k)\in\Horn(B[\gamma],\kappa)$ where
$q_k^\gamma,p_k^\gamma$ are computed from $B[\gamma]$ rather than from
$B$. Running over (the symbols of) all possible such $P$ corresponds
to running over (the symbols of) all elements of
$\Horn(B[\gamma],\kappa)$. Since $\Horn(B[\gamma],\kappa)$ and
$\sHorn(B[\gamma],\kappa)$ agree on $(\CC^*)^\gamma$, and since
$\sHorn(B[\gamma],\kappa)$ is holonomic, the characteristic variety
$C$ meets $T_\gamma$ in a set of dimension at most $m$, which is the sum of
$|\gamma|$ (the dimension of the characteristic variety of
$\sHorn(B[\gamma],\kappa)$) and $m-|\gamma|$ (the dimension of the
fiber of the projection $T^*_\gamma\to T^*\CC^\gamma$). 
\end{proof}

When $B$ has only two rows, there is no difference between toral and
completely toral parameters. This fact provides a holonomicity characterization for
bivariate Horn $D_\barZ$-modules.

\begin{corollary}
\label{coro:bivariateHolonomic}
If $B \in \ZZ^{n\times 2}$, then the following are equivalent:
\vspace*{-2mm}
\begin{enumerate}
\item The system $D_\barZ/\Horn(B,\kappa)$ is holonomic.
\item The parameter $A \kappa$ is toral for $B$. 
\item The system $D_\barX/(I(B)+\<E-A\kappa\>)$ is holonomic.
\end{enumerate}
\end{corollary}

\section{Torus invariants and the Horn--Kapranov uniformization}
\label{sec:binomial Sing}

In this section, we consider the singular locus of the image of an
$A$-hypergeometric system under $\Pi^{\tilde A}_B$, still following
Convention~\ref{conv:BtildeACK}.  
We show that our framework for constructing $\Pi^{\tilde A}_B$ 
can be thought of as a generalization of 
Kapranov's ideas in~\cite{horn-kapranov}.  

In this direction, we restrict our attention to the case  
that the $\QQ$-rowspan of $A$ contains $\boldone_n$, so that the $A$-hypergeometric system 
$D_\barX/( I_A+\<E-\beta\>) = D_\barX/H_A(\beta)$ is regular holonomic
for all $\beta$~\cite{hotta,schulze-walther-duke}. 
This allows us to apply Proposition~\ref{prop:characteristic variety Pi},
which provides a precise description of the singular locus  
$\Sing{\Pi^{\tilde A}_B(D_\barX/H_A(\beta))}$. 
Namely, it is the union of  
$\Sing{\Delta_B^{\tilde A}(D_\barX/H_A(\beta))}$ with 
the coordinate hyperplanes in $\barZ$. 
The former is a geometric quotient by the torus $T$ of $\Sing{D_\barX/H_A(\beta)}$. 
We note that in~\cite{HT-sing lauricella}, Gr\"obner bases and
$D$-module theory are used to determine the singular locus of the
Lauricella $F_C$ system, as well as that of its associated binomial
$D$-module.   

The \emph{secondary polytope of $A$} is the Newton polytope of the
principal $A$-determinant, which is the defining equation for the codimension one part of 
$\Sing{D_\barX/H_A(\beta)}$. Since the principal $A$-determinant is
$A$-graded, its Newton polytope lies in an $m$-dimensional subspace of
$\CC^n$. Thus, as a direct consequence of
Proposition~\ref{prop:characteristic variety Pi}, we have the following result. 

\begin{corollary}
\label{cor:sing Pi binomial}
The Newton polytope of the defining polynomial of the codimension one part of 
$\Sing{\Pi^{\tilde A}_B(D_\barX/H_A(\beta))}$ 
is the secondary polytope of $A$ 
viewed in $\RR^m$. 
\qed
\end{corollary}

By following essentially the same argument as the 
one given in the proof in~\cite[Theorem~2.1.a]{horn-kapranov}, we recover
the Horn--Kapranov uniformization of the $A$-discriminant from our
setup for the construction of $\Pi^{\tilde A}_B$.   
Let $[n] \defeq \{1,\dots,n\}$, and 
let $C_{[n]}$ be the closure of the conormal space 
to the $T$-orbit of the point $\boldone_n$. 
Then $C_{[n]}$ has dimension $n$ and has defining ideal equal to the radical of 
$\<\gr^F(E), \gr^F(I_A) \>$.

Since we have assumed that the rational row span of $A$ contains the vector
$\boldone_n$, the $A$-discriminant is the defining polynomial of the
projection onto the $x$-coordinates of $C_{[n]} \minus \Var(\xi_1,\dots,\xi_n)$, 
as long as this projection is a hypersurface. 
Otherwise, the $A$-discriminant is defined to be the polynomial $1$. 
In the hypersurface case, the \emph{reduced $A$-discriminant} is the 
polynomial $\nabla_A$ given by saturating $x_1\cdots x_n$ out of the
principal $A$-determinant, and thus cuts out the closure of the intersection of the
$A$-discriminantal hypersurface with $(\CC^*)^n$. 
Note that the $A$-discriminant is $A$-graded, and thus $\nabla_A$ is
invariant with respect to the torus action.

Since the algebraic counterpart of projection is elimination, we see
that $\nabla_A$, considered as a polynomial in $2n$ variables,
vanishes on  
$C_{[n]} \minus \Var(\<x_1\cdots x_n\> \cap \<\xi_1,\dots,\xi_n\> )$, 
and in particular on $C_{[n]} \minus \Var(x_1\cdots x_n \cdot \xi_1\cdots\xi_n)$. 
Write $\nabla_A$ as a finite sum $\sum_{w\in
  \ZZ^m} \lambda_w x^{Bw}$. If $(\bar{x}, \bar{\xi})$ belongs to $C_{[n]} \minus
\Var(x_1\cdots x_n \cdot \xi_1\cdots\xi_n)$, then 
\[
0= \nabla_A(\bar{x},\bar{\xi}) =  \sum_{w\in \ZZ^m} \lambda_w \bar{x}^{Bw} = 
\sum_{w\in \ZZ^m} \lambda_w \bar{x}^{Bw} \bar{\xi}^{Bw} = 
\sum_{w \in \ZZ^m} \lambda_w (\bar{x}\bar{\xi})^{Bw},
\]
as $\xi^{Bw} = 1$ on $C_{[n]} \minus \Var(x_1\cdots x_n \cdot
\xi_1\cdots\xi_n)$. Dehomogenizing, we obtain  
\[
\dkap\left(\sum_{w \in \ZZ^m} \lambda_w (x\xi)^{Bw}\right) = 
\sum_{w \in \ZZ^m} \lambda_w (Bz\zeta)^{Bw} = \sum_{w\in \ZZ^m}
\lambda_w((Bz\zeta)^B)^w, 
\] 
where, for $x \in X$ and $b_1,\dots,b_m$ the columns of $B$,
we write $x^B \defeq (x^{b_1},\dots,x^{b_m})$, as in \S\ref{sec:E-beta}.

Replacing $z\zeta$ by $(s_1,\dots,s_m)$, it is now easy to see that
the dehomogenized reduced $A$-discri\-mi\-nant is the defining
equation for the variety with the desired parametrization:  
\[
\left( \CC^m \minus \bigcup_{i=1}^n \Var((Bs)_i) \right)\ni s \mapsto (Bs)^B, 
\]
known as the Horn--Kapranov uniformization.

\raggedbottom
\def\cprime{$'$} \def\cprime{$'$}
\providecommand{\href}[2]{#2}
\end{document}